\documentclass[11pt]{article}
\newcommand{\blind}{1}
\RequirePackage{graphicx}
\usepackage{amsfonts}
\usepackage{bbm}
\usepackage{comment}
\usepackage{natbib}
\usepackage{mathtools}
\usepackage[left=1in,right=1in,top=1in,bottom=1in]{geometry}
\usepackage{pgfplots}
\pgfplotsset{compat=1.16} 
\usepackage[all]{nowidow}
\usepackage[utf8]{inputenc}
\usepackage{tikz}
\usepackage{multicol}
\usepackage{algpseudocode,algorithm,algorithmicx}

\usepackage{scalerel}
\usepackage[normalem]{ulem}

\usepackage[english]{babel}
\usepackage{caption}

\usepackage{mathdots}
\usepackage{yhmath}
\usepackage{cancel}
\usepackage{color}
\usepackage{array}
\usepackage{multirow}
\usepackage{enumerate}
\usepackage{gensymb}
\usepackage{tabularx}
\usepackage{booktabs}
\usetikzlibrary{fadings}
\usetikzlibrary{patterns}
\usetikzlibrary{shadows.blur}

\usepackage{amssymb}
\usepackage{amsmath}    
\usepackage{dirtytalk}
\usepackage{amsthm}

\newcommand{\R}{\mathbb{R}}

\newcommand{\one}{\mathbf 1}

\newcommand{\I}{\mathcal{I}}

\newcommand{\cI}{\mathcal{I}}

\newcommand{\cS}{\mathcal{S}}

\newcommand{\bP}{\mathbb{P}}
\newcommand{\bR}{\mathbb{R}}

\DeclareMathOperator*{\argmin}{arg\,min}

\theoremstyle{plain}

\newtheorem{theorem}{Theorem}[section]
\newtheorem{lemma}[theorem]{Lemma}
\newtheorem{proposition}{Proposition}
\theoremstyle{definition}
\newtheorem{definition}[theorem]{Definition}
\newtheorem{assumption}{Assumption}

\newtheorem{remark}{Remark}
\theoremstyle{remark}

\usepackage[nodisplayskipstretch]{setspace}

\usepackage[colorlinks,citecolor=blue,urlcolor=blue]{hyperref}
\begin{document}
\date{}

\def\spacingset#1{\renewcommand{\baselinestretch}%
{#1}\small\normalsize} \spacingset{1}


\if1\blind
{
  \title{An Exact Pointwise Characterization for Total Variation Denoising in Quantile Regression}
  \author{Deep Ghoshal and Sabyasachi Chatterjee \\
    Department of Statistics, University of Illinois at Urbana-Champaign
  }
  \maketitle
}
\fi

\bigskip


\begin{abstract}
Total variation denoising (TVD) is a classical method for denoising and curve fitting, yet an explicit pointwise description of its fitted values has only recently been established in the mean regression setting by \cite{chatterjee2026minmaxtrendfilteringgeneralizations}. This raises the question of whether a similar representation holds for quantile regression.

 We answer this question affirmatively by deriving an exact minmax/maxmin representation for the quantile TVD estimator, providing a complete pointwise characterization of its solution set. Given that the quantile TVD estimator is generally non-unique, the existence of such a representation is perhaps surprising. We show that the set of admissible fitted values at any location forms a compact interval, whose endpoints are characterized exactly by minmax/maxmin functionals of local order statistics over nested intervals.

 We next develop several structural properties of the quantile TVD solution set. First, the solution set is closed under coordinatewise maximum and minimum, guaranteeing the existence of extremal elements—upper and lower envelope solutions. Second, this reveals that quantile TVD is intrinsically non-crossing across quantile levels when a common tuning parameter is used. We prove this is driven by submodularity of the total variation penalty, and show that any penalized quantile regression estimator with a submodular penalty enjoys this property.

 From an estimation error perspective, our representation enables a refined pointwise analysis via a transparent local bias-variance decomposition, facilitating new pointwise risk bounds and near-optimal rates for locally H\"{o}lder smooth functions. Our results hold under heavy-tailed noise (e.g., Cauchy) and substantially extend existing guarantees beyond locally constant signals. Altogether, these results advance the theory of quantile TV regression via exact pointwise min-max representations. 
\end{abstract}

\noindent%

\spacingset{1}
\section{Introduction}\label{sec:intro}
\subsection{Pointwise Characterization of Quantile Total Variation Denoising}\label{sec:intro_qtltvd}
In this article, we revisit the univariate quantile total variation denoising (TVD) estimator in the setting of nonparametric quantile regression. Since the seminal work of \cite{koenker1978regression} on parametric quantile regression, the field has witnessed extensive development in both statistics and economics. Owing to its inherent robustness, quantile regression has become a powerful tool for modeling and prediction. For foundational contributions to nonparametric quantile regression, see, for example, \cite{utreras1981computing, cox1983asymptotics, koenker1994quantile, chaudhuri1991nonparametric}.

Total variation regularization has a long and rich history in signal processing and statistics, beginning with the seminal work of~\cite{rudin1992nonlinear} on image denoising and its statistical counterpart in the fused lasso~\cite{tibshirani2005sparsity}. A continuous analogue of the TVD estimator—penalizing the $L_1$ norm of the derivative—was introduced earlier in the statistics literature by \cite{mammen1997locally} under the name \emph{locally adaptive regression splines}. In this work, we focus on total variation regularization in the context of univariate quantile regression.

The quantile TVD estimator, for any vector $y \in \R^n$, is defined as any element of the solution set
\begin{equation}\label{eqn:defn_qtl_TVD}
\hat{\theta}^{\tau}\in\cS^{\tau}:=\argmin_{\theta \in \bR^n}\left\{\sum_{i=1}^{n}\rho_{\tau}(y_i-\theta_i)+\lambda\,TV(\theta)\right\},
\end{equation}
where $\rho_{\tau}(x)=\max\{\tau x, (\tau-1)x\}$ is the standard convex piecewise linear quantile loss, $\lambda>0$ is a tuning parameter, and $TV(\theta):=\sum_{i=1}^{n-1}|\theta_{i+1}-\theta_i|$ denotes the total variation of $\theta$. Here, $\argmin$ denotes the (possibly non-singleton) set of minimizers.

A key feature of the TV penalty is that it promotes sparsity in first-order differences, thereby producing piecewise constant estimates. This property makes univariate TVD particularly well-suited for detecting structural breaks and denoising signals with piecewise constant structure.

The statistical properties of the mean regression version of TVD—obtained by replacing the quantile loss $\rho_{\tau}$ with the squared loss $x \mapsto x^2$—have been extensively studied; see, for example, \cite{Tibshirani2013AdaptivePP, tibshirani2022divided, guntuboyina2020adaptive, dalalyan2017prediction, harchaoui2010multiple, lin2017sharp, ortelli2018total, ortelli2021prediction, sadhanala2019additive, sadhanala2024multivariate, chatterjee2026minmaxtrendfilteringgeneralizations}. In contrast, the quantile version of TVD has received comparatively less attention in the literature.

In the setting of quantile regression, a class of related estimators, known as \emph{quantile smoothing splines}, was introduced by \cite{koenker1994quantile}. These estimators are curve-based and are defined as
\[
\mathop{\mathrm{minimize}}_{f \in \mathcal{F}} \left\{ \sum_{i=1}^n \rho_\tau \bigl( y_i - f(x_i) \bigr) + \lambda \left( \int_0^1 |f''(x)|^p \, dx \right)^{1/p} \right\},
\]
where $0 < x_1 < \dots < x_n < 1$, $\lambda > 0$ is a tuning parameter, $p \ge 1$, and $\mathcal{F}$ is a suitable function class. 

If one replaces the second derivative by the first derivative and sets $p=1$, the resulting estimator can be viewed as a continuous analogue of the quantile TVD estimator defined in~\eqref{eqn:defn_qtl_TVD}. In fact, when $p=1$, the corresponding quantile smoothing spline coincides with the quantile version of locally adaptive regression splines of order $2$, introduced by \cite{mammen1997locally}.

To the best of our knowledge, the quantile TVD estimator was first explicitly introduced by \cite{li2007analysis} in the context of detecting DNA copy number variations in genomic data. Since then, only a limited number of works have studied quantile TVD; see, for example, \cite{madrid2022risk, Brantley2020, feng2024deep, zhang2025quantile, madrid2024quantile}.


Despite the long history of the univariate TVD estimator in the mean regression setting, an explicit pointwise description of its fitted values remained unavailable until recently. In a recent work, \cite{chatterjee2026minmaxtrendfilteringgeneralizations} derived an exact pointwise representation of the estimator, showing that each fitted value can be expressed as a min--max/max--min of simple functions of local averages over nested intervals containing the target point. This representation has enabled the analysis of pointwise risk for the TVD estimator, which was previously inaccessible using existing techniques.

A natural question is whether a similar representation holds for the quantile TVD estimator. However, a key distinction in the quantile setting is that the estimator is generally non-unique, as the objective function in~\eqref{eqn:defn_qtl_TVD} may admit multiple minimizers. In view of this intrinsic non-uniqueness, it is not \emph{a priori} clear in what sense a comparable pointwise characterization could hold.

The main contribution of this article is to resolve this question in the affirmative. We establish a novel, exact pointwise characterization of the quantile TVD estimator. Specifically, we show that the set of feasible fitted values at any given location forms a compact interval in $\mathbb{R}$, and we characterize the endpoints of this interval explicitly via min--max/max--min functionals of local order statistics over nested intervals containing the target point.

 In order to state the central result of this paper, we first set up the following convention about order statistics. For two integers $a\leq b$, let us denote the discrete interval $\{a,\ldots,b\}$ by $[a:b]$ and for any integer $n$, let us denote the set $\{1,\ldots,n\}$ by $[n]$.  

\begin{definition}\label{defn:Y_I,k}
    For any discrete interval $I \subseteq [1:n]$ and any integer $k$, let us define $y_{I,(k)}$ as
    \[y_{I,(k)}:=\begin{cases}
     k\text{-th smallest element of }y_I & \text{ if } 1\leq k \leq |I|;\\
     \infty & \text{ if } k\geq |I|+1;\\
     -\infty & \text{ if } k\leq 0.
 \end{cases}\]
\end{definition}

We now present the main contribution of this paper in the following theorem.

\begin{theorem}\label{thm:minmax_bound}
    Fix any $\tau\in (0,1)$, $\lambda\geq 0$ and a location $i\in [n]$. Let us define $\cS^{\tau}_i$ to be the set of all feasible values of the estimator defined in \eqref{eqn:defn_qtl_TVD} at the location $i$, i.e., 
    \[\mathcal{S}^{\tau}_i:=\{\theta_i\;:\theta \in \cS^{\tau}\}.\]
    Let $\I$ denote the set of all discrete sub-intervals of $[1:n]$. 
    Then, the following hold:
    \begin{enumerate}
        \item [(1)] (Exact Pointwise Characterization) \(\cS^{\tau}_i\) is a compact interval in $\bR$ given by \([L_i^{\tau},U_i^{\tau}]\), where
        \begin{equation} \label{eqn:minmax_bound_qtl}
  L^{\tau}_i:=\max_{J\in \I: i \in J}\min_{I\in \I: I\subseteq J,\, i \in I}\,y_{I,(\lceil l^{\tau}_{I,J}\rceil)},\qquad U^{\tau}_i:= \min_{J\in \I: i \in J}\max_{I\in \I: I\subseteq J,\, i \in I}\, y_{I,(\lfloor u^{\tau}_{I,J}\rfloor+1)},
 \end{equation}
  where the quantities $u^{\tau}_{I,J}$ and $l^{\tau}_{I,J}$ are defined in Definition~\ref{defn:u_ij_l_ij}, and $y_{I,(k)}$ is defined in Definition~\ref{defn:Y_I,k}.
 \item [(2)] (Lattice Property) The solution set $\mathcal{S}^{\tau}$ forms a lattice, i.e., it is closed under coordinatewise maximum and minimum. Consequently, there exists a solution of the objective function that attains the value of the upper (or lower, respectively) envelope $U^{\tau}$ (or $L^{\tau}$, respectively) at all locations simultaneously, i.e., 
 $\exists\;\; \hat{\theta}^{Upper}(\text{ or }\hat{\theta}^{Lower}\text{, respectively) }\in \cS^{\tau}\text{ such that }$
 $$\hat{\theta}^{Upper}_i(\text{ or }\hat{\theta}^{Lower}\text{, respectively) }=U^{\tau}_i(\text{ or }L^{\tau}_i,\text{ respectively) },\;\;\forall\;i\in [n].$$
 \item [(3)] (Non-Crossing Property) Let $0<\tau_1< \tau_2<1$. If $\hat{\theta}^{\tau_1}\in\cS^{\tau_1}$ and $\hat{\theta}^{\tau_2}\in\cS^{\tau_2}$, then for any $i\in [n]$, it holds that $\hat{\theta}^{\tau_1}_i\leq \hat{\theta}^{\tau_2}_i$.
    \end{enumerate}   
\end{theorem}

\begin{definition}\label{defn:u_ij_l_ij}
For $I,J\in \I$ such that $I\subseteq J$, we define
    \[u^{\tau}_{I,J}:=\tau|I|-2\lambda C_{I,J};\]
\[l^{\tau}_{I,J}:=\tau|I|+2\lambda C_{I,J},\]
where $C_{I,J}$ is as defined below in Definition~\ref{defn:C_{IJ}}.
\end{definition}

\begin{definition}\label{defn:C_{IJ}}
    Fix a discrete interval $J \subseteq[n]$ and a subinterval $I \subseteq J$. Write $J= \left[j_1: j_2\right]$ and $I=[s: t]$. Throughout, $I \subset J$ denotes that $I$ lies strictly in the interior of $J$, that is, $I$ does not contain either boundary point of $J$. Depending on whether $J$ contains none, one or both global boundary points $\{1, n\}$, $C_{I,J}$ is defined as follows:
  \begin{enumerate}
    \item If $1 < j_1 \leq j_2 < n$, then
    \[
    C_{I, J} = \begin{cases} 
    1, & \text{if } I \subset J \\ 
    -1, & \text{if } I = J \\ 
    0, & \text{otherwise}.
    \end{cases}
    \]

    \item If $1 = j_1 \leq j_2 < n$, then
    \[
    C_{I, J} = \begin{cases} 
    1, & \text{if } I \subset J \\ 
    -1/2, & \text{if } I = J \\ 
    1/2, & \text{if } 1 = s \leq t < j_2 \\ 
    0, & \text{if } 1 < s \leq t = j_2. 
    \end{cases}
    \]

    \item If $1 < j_1 \leq j_2 = n$, then
    \[
    C_{I, J} = \begin{cases} 
    1, & \text{if } I \subset J \\ 
    -1/2, & \text{if } I = J \\ 
    1/2, & \text{if } j_1 < s \leq t = n \\ 
    0, & \text{if } j_1 = s \leq t < n. 
    \end{cases}
    \]

    \item If $J = [1 : n]$, then
    \[
    C_{I, J} = \begin{cases} 
    1, & \text{if } I \subset J \\ 
    0, & \text{if } I = J \\ 
    1/2, & \text{otherwise}. 
    \end{cases}
    \]
\end{enumerate}
\end{definition}

Let us explain the meaning of Theorem~\ref{thm:minmax_bound}. The first part of the theorem conveys a two fold message. Firstly, it says that for any minimizer $\hat{\theta}^{\tau}$ of the objective function in \eqref{eqn:defn_qtl_TVD}, the fitted value at any given location $i\in [n]$ satisfies $L^{\tau}_i \leq \hat{\theta}^{\tau}_i \leq  U^{\tau}_i.$ Secondly, for any value $v \in [L^{\tau}_i,U^{\tau}_i]$, there exists a minimizer $\hat{\theta}^{\tau}\in \cS^{\tau}$ such that $\hat{\theta}^{\tau}_i=v$. Equivalently, the set $\mathcal{S}^{\tau}_i$ is exactly equal to the interval $[L^{\tau}_i,U^{\tau}_i].$ This characterization in \eqref{eqn:minmax_bound_qtl} holds for any data vector $y\in \bR^n$, any location $i \in [n]$, any tuning parameter $\lambda\geq 0$ and any quantile level $\tau \in (0,1).$

The main point is that we obtain an explicit characterization of the upper (and similarly, the lower) envelope $U_i^{\tau}$ as a min--max functional of local order statistics. Concretely, for a fixed location $i$, we consider an outer interval $J$ containing $i$ and an inner sub-interval $I \subseteq J$ that also contains $i$. Each such pair of nested intervals $(J,I)$ yields a local estimate given by
\[
y_{I,(\lfloor u^{\tau}_{I,J}\rfloor+1)},
\]
the $(\lfloor u^{\tau}_{I,J}\rfloor+1)$-th order statistic of the observations restricted to the interval $I$. The upper envelope is then obtained by taking the maximum over all admissible inner intervals $I \subseteq J$ followed by the minimum over all outer intervals $J$ containing $i$. The lower envelope is defined analogously via a corresponding max--min operation.

Intuitively, since the goal is to estimate the $\tau$-th quantile, it is natural to consider empirical $\tau$-quantiles over local intervals $I$. However, our theorem reveals that the quantile TVD estimator instead operates with \emph{adjusted} local quantile levels. These adjustments are governed by the constants $C_{I,J}$; more precisely, the relevant local estimates are given by order statistics at levels $\tau|I| \pm 2\lambda C_{I,J}$ (up to integer rounding) within the interval $I$.

The constants $C_{I,J} \in \{\pm 1, \pm \tfrac{1}{2}, 0\}$ take only a few discrete values and are exactly the same as those appearing in the mean TVD representation~\cite{chatterjee2026minmaxtrendfilteringgeneralizations}. This is because these \emph{adjustment constants} arise fundamentally from the TV penalty itself. As is evident from the proof, their form is dictated entirely by the structure of the penalty, indicating that $C_{I,J}$ are intrinsically tied to the TV regularization.

As noted earlier, part~(1) of Theorem~\ref{thm:minmax_bound} implies that for any fixed location $i \in [n]$, there exists a solution $\hat{\theta}^{\max}$ (respectively, $\hat{\theta}^{\min}$) in $\cS^{\tau}$ such that $\hat{\theta}^{\max}_i = U_i$ (respectively, $\hat{\theta}^{\min}_i = L_i$). This naturally raises the question of whether one can find a \emph{single} solution $\hat{\theta} \in \cS^{\tau}$ that simultaneously attains these extremal values across all coordinates, i.e., $\hat{\theta}_i = U_i$ (or $L_i$) for every $i \in [n]$.

The second part of Theorem~\ref{thm:minmax_bound} answers this question in the affirmative. It shows that the solution set $\cS^{\tau}$ is closed under coordinatewise maximum and minimum (see Lemma~\ref{lemma:minmax_preserves_qtltvd_tvdonly} for a precise statement). As a consequence, $\cS^{\tau}$ forms a lattice, which in turn guarantees the existence of \emph{maximal} and \emph{minimal} elements. In particular, there exists a solution in $\cS^{\tau}$ that simultaneously attains the upper envelope $U_i$ at all locations $i \in [n]$, and similarly, a solution that attains the lower envelope $L_i$ for all $i \in [n]$.

Although the existence of maximal and minimal elements in $\mathcal{S}^{\tau}$ might suggest that the solution set fills out the entire rectangle $\prod_{i=1}^n [L_i, U_i]$, this intuition is incorrect, as can be seen from counterexamples even when $n = 3$.

The third part of the theorem reveals a striking property: solutions corresponding to different quantile levels cannot cross. We refer to Section~\ref{sec:intro_noncrossing} for a more detailed discussion of this phenomenon.

Figure~\ref{fig:qtltvd_envelope} illustrates an example of quantile regression for $\tau_1 = 0.25$ and $\tau_2=0.75$ using quantile TVD on a synthetic dataset of size $20$. 
\begin{figure}
\centering
\includegraphics[scale=0.7]{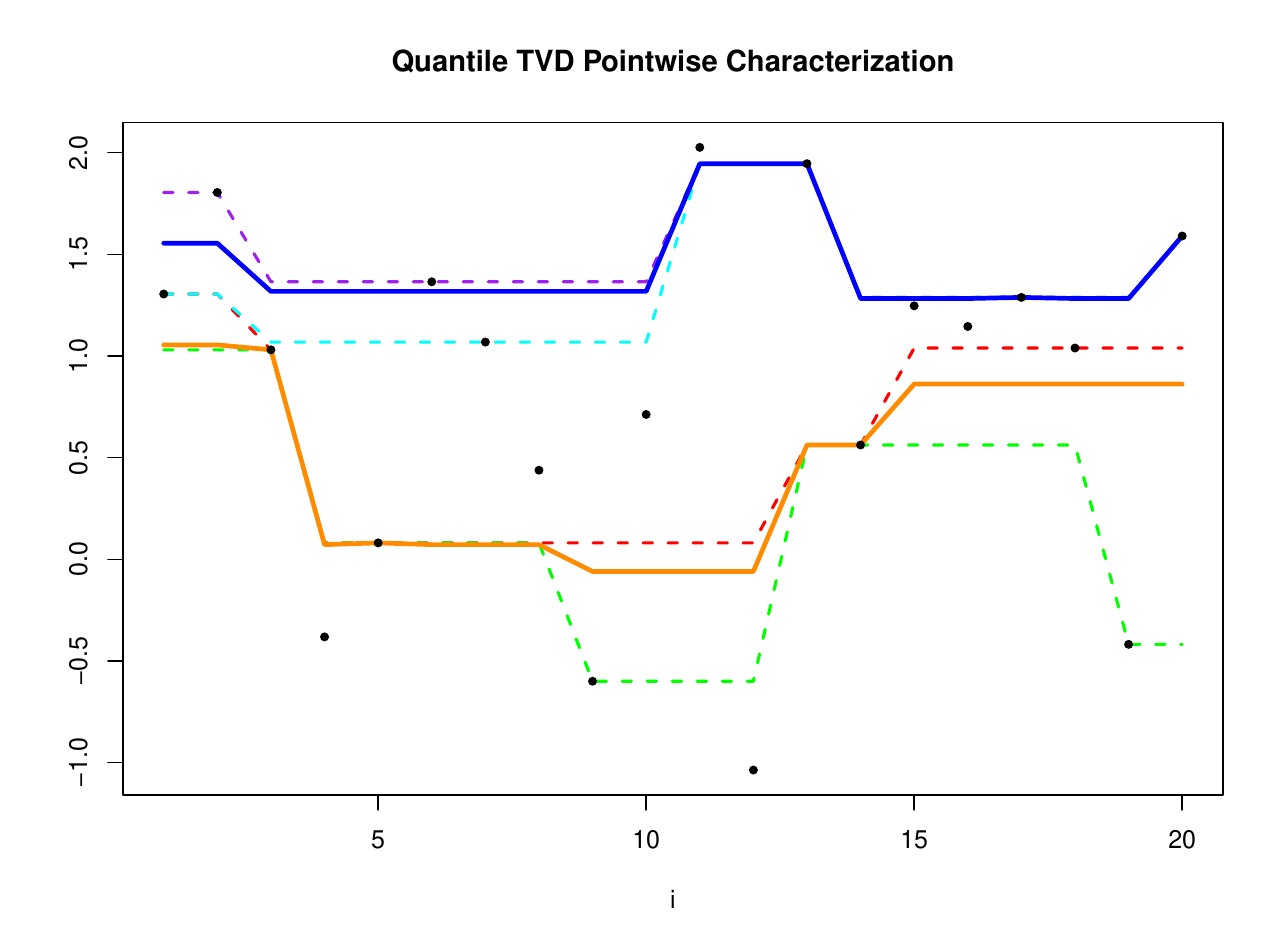}
\caption{Quantile TVD estimators independently estimated for $\tau_1=0.25$ (solid orange curve) and $\tau_2=0.75$ (solid blue curve) with $\lambda = 0.5$ , computed via an iterative algorithm. For $\tau_1=0.25$, the dashed red and green curves represent the boundary solutions $\hat{\theta}^{Upper}_i = U^{\tau_1}_i$ and $\hat{\theta}^{Lower}_i = L^{\tau_1}_i$, respectively, for all $i$. For $\tau_2=0.75$, these boundary solutions are represented by the dashed purple and cyan lines, respectively.}

\label{fig:qtltvd_envelope}
\end{figure}
\medskip

The fact that the set of feasible values $\cS^{\tau}_i$ of the estimator at a location $i$ forms a compact interval follows from a basic real analysis argument. The key technical contribution of this work lies in the explicit characterization of the endpoints of this interval. As a first step, we show that for any minimizer $\hat{\theta}^{\tau}$ of the objective function, the value $\hat{\theta}^{\tau}_i$ necessarily lies in the interval $[L_i, U_i]$. 

This part of the proof follows the general template of \cite{chatterjee2026minmaxtrendfilteringgeneralizations}, relying on the optimality conditions of the associated convex optimization problem and the development of a so-called \emph{interval identity} (Lemma~\ref{lem:qtvd-interval}).

However, this alone does not suffice to conclude that $L_i$ and $U_i$ are the exact endpoints of $\cS_i$. It remains to show that these bounds are actually attained. To this end, we consider extremal solutions $\hat{\theta}^{\tau,\max}$ and $\hat{\theta}^{\tau,\min}$, whose values at the target location $i$ are, respectively, the largest and smallest possible among all minimizers. By exploiting this maximality (or minimality) and developing a suitable perturbation argument, we establish that $\hat{\theta}^{\tau,\max}_i = U_i$ and $\hat{\theta}^{\tau,\min}_i = L_i$.

This step constitutes a key technical innovation of the paper. It is both necessary and nontrivial in the quantile setting, in contrast to the mean case treated in \cite{chatterjee2026minmaxtrendfilteringgeneralizations}, where uniqueness of the estimator obviates the need for such an argument. The full details of the proof are provided in Section~\ref{sec:minmax}.

\subsection{Monotonicity of Quantile TVD Estimates}\label{sec:intro_noncrossing}
 A well-known phenomenon in quantile regression is that when quantile curves corresponding to different levels $\tau_1 < \tau_2$ are estimated independently, the resulting estimates may cross; see Figure~\ref{fig:qtlTF_Crossing}. Such behavior violates the natural monotonicity requirement across quantile levels and is commonly referred to as the \emph{quantile crossing problem}. This issue was first noted in \cite{bassett1982empirical}, and has since been extensively studied; see, for example, \cite{he1997quantile, chernozhukov2010quantile, JMLR:v7:takeuchi06a, bondell2010noncrossing}.

A common approach to address this issue is to explicitly impose non-crossing constraints within the optimization problem or to apply post-processing procedures, such as rearrangement or sorting; see, for instance, \cite{JMLR:v7:takeuchi06a, bondell2010noncrossing}. In the context of quantile TVD, a similar strategy is often adopted in practice, where non-crossing is enforced as a hard constraint; see the \texttt{quantgen} R package~\cite{quantgen-package}.
 
Figure~\ref{fig:qtltvd_envelope} empirically demonstrates that the minimal solution for a higher quantile level never intersects the maximal solution for a lower quantile level. Given the exact pointwise characterization developed in part $(1)$ of Theorem~\ref{thm:minmax_bound}, this translates directly to the estimated quantile curves being non-crossing. Upon observing the same phenomenon across numerous other simulation studies on quantile TVD, we naturally wondered whether this phenomenon extends to more general penalties, or is it specific to the TV penalty? For instance, one may ask whether higher-order extensions such as quantile trend filtering (see \cite{madrid2022risk}) also automatically enforce non-crossing.

Our empirical investigations suggest that this is not the case. In particular, second-order quantile trend filtering can produce crossing solutions, as illustrated in Figure~\ref{fig:qtlTF_Crossing}.

\begin{figure}[!h]
\centering
\includegraphics[scale=0.6]{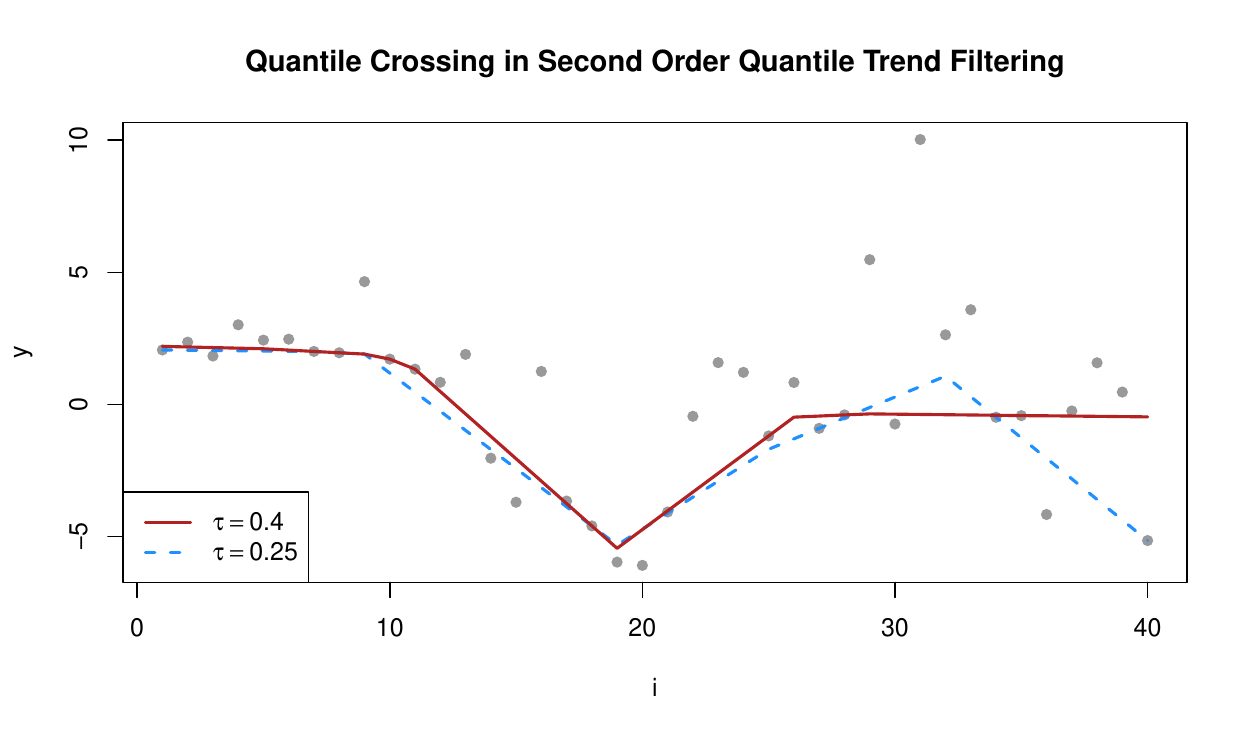}
\caption{Quantile crossing in second-order quantile trend filtering.}
\label{fig:qtlTF_Crossing}
\end{figure}

Motivated by these observations, we also wondered whether it was indeed possible to mathematically establish non-crossing of quantile TVD. Theorem~\ref{thm:minmax_bound} answers this question in the affirmative. Part~(3) shows that the estimator is intrinsically non-crossing: for any two quantile levels $\tau_1 < \tau_2$, the corresponding quantile TVD estimates cannot cross when a common tuning parameter is used. To the best of our knowledge, this property has not been previously identified and distinguishes quantile TVD from many other quantile regression methods.

An immediate implication of this result is that no additional constraints or post-processing steps are required to enforce non-crossing when estimating multiple conditional quantile curves simultaneously using quantile TVD with a common tuning parameter.

In further pursuit of understanding the exact structural property of the TVD estimator that drives this non-crossing phenomenon, we discovered that the key mechanism is the \emph{submodularity} of the TV penalty. This is indeed the key technical ingredient in the proof of Theorem~\ref{thm:minmax_bound}, part $(3)$. Building on this insight, we establish a general result showing that \textit{any penalized quantile regression estimator with a submodular penalty is intrinsically non-crossing}. To the best of our knowledge, such a result has not been previously identified and may be of independent interest. We defer the formal definition of submodularity and the proof of this result to Section~\ref{sec:monotonicity}.

\subsection{Local Risk Analysis of Quantile TVD}\label{sec:intro_localrisk}

Another key contribution of this work is the use of the minmax/maxmin characterization to derive pointwise risk bounds for the quantile TVD estimator. We analyze the estimator under the quantile sequence model
\begin{equation}\label{eqn:qtl_seq_model}
    y_i = \theta_i^* + \epsilon_i, \qquad i=1,\dots,n,
\end{equation}
where the errors $\epsilon_i$ are independent with $\tau$-th quantile equal to zero, and $\theta_i^*$ denotes the $\tau$-th quantile of the signal at location $i$. This model naturally arises in nonparametric quantile regression with fixed design points; for simplicity, we assume an equally spaced grid on $[0,1]$. Letting $\theta_i^* = f^*(i/n)$ for an unknown function $f^*$, the goal is to estimate $f^*$ under suitable structural assumptions using total variation regularization.

The minmax/maxmin representation gives rise to a local, nonstandard, and multiscale bias--variance decomposition of the pointwise estimation error. Leveraging this decomposition, we derive pointwise risk bounds for the quantile TVD estimator under the model~\eqref{eqn:qtl_seq_model}. These bounds hold for arbitrary signals and simultaneously for all locations $i \in [n]$ with polynomially high probability.

Notably, in contrast to the mean regression setting, our bounds remain valid under heavy-tailed noise distributions, including the Cauchy distribution. As a consequence, we obtain local rates of convergence when the true quantile function is locally H\"older smooth. The resulting bounds exhibit an explicit dependence on the tuning parameter $\lambda$, thereby elucidating both how the risk varies with $\lambda$ and how the optimal choice of $\lambda$ depends on the local smoothness of $f^*$.

To the best of our knowledge, the only comparable analyses of quantile TVD are \cite{madrid2022risk} and \cite{zhang2023element}. The former studies global risk under Huber loss, while the latter derives pointwise risk bounds restricted to piecewise constant signals. In contrast, our results establish pointwise risk bounds for general H\"older smooth signals, thereby substantially extending the existing theory. We present the detailed bias--variance decomposition and the resulting pointwise risk bounds in Section~\ref{sec:ptwise_error}.

\subsection{Outline of the Paper}
In Section~\ref{sec:minmax}, we present a detailed proof of Theorem~\ref{thm:minmax_bound}. As this theorem constitutes the central result of the paper, its proof forms the core technical contribution and is therefore presented first. 

In Section~\ref{sec:monotonicity}, we develop the connection to submodularity. We begin by reviewing basic properties of submodular functions and then prove that any penalized quantile regression estimator with a submodular penalty is intrinsically non-crossing across quantile levels, provided a common tuning parameter is used. We also discuss additional examples of submodular penalties beyond total variation.

In Section~\ref{sec:ptwise_error}, we show how the min--max/max--min characterization naturally yields a local bias--variance decomposition of the pointwise estimation error, and we use this decomposition to derive new pointwise risk bounds for quantile TVD under general H\"older smoothness assumptions.

Finally, in Section~\ref{sec:conclusion}, we summarize our contributions.

\section{Proof of Theorem~\ref{thm:minmax_bound}}\label{sec:minmax}
\subsection{Proof Outline}
While the proof of Theorem~\ref{thm:minmax_bound} constitutes the central technical contribution of this work, it is lengthy and involved. We therefore begin with a brief outline of the main ideas, highlighting several key technical components. For notational simplicity, we suppress the superscript $\tau$ whenever it is clear from the context.
\begin{enumerate}
    \item Fix $i\in [n]$. Using a basic real analysis argument, we first identify that the set $\cS_i$ is a compact interval in $\bR$.
    \item Then, we utilize the optimality conditions related to the convex optimization problem in \eqref{eqn:defn_qtl_TVD} and identify an \textit{interval averaging identity} concerning corresponding dual variables, see Lemma~\ref{lem:qtvd-interval}.
    \item For \textit{any} $\hat{\theta}\in \cS$ and \textit{any} $J = [a:b]$ containing $i$, if one considers $I=[c:d]\subseteq J$ such that $I$ contains $i$ and it is the \textit{largest subinterval of }$J$ \textit{such that } $\hat{\theta}_i$ is a local minimum, then using the interval identity in Lemma~\ref{lem:qtvd-interval} we show that
\[
\hat{\theta}_i\leq y_{I,(\lfloor u_{I,J}\rfloor+1)}\leq \max_{I\in \I: I\subseteq J,\, i \in I}\, y_{I,(\lfloor u_{I,J}\rfloor+1)}.
\]
Since the choice of $J$ was arbitrary, we conclude $\hat{\theta}_i\leq U_i$. We then argue via an anti-symmetry argument that this automatically implies $L_i\leq \hat{\theta}_i$. This proves that $\cS_i\subseteq [L_i,U_i].$

\item We next consider an element $\hat{\theta}^{\max}\in \cS$ such that $\hat{\theta}^{\max}_i$ attains the maximum possible value at location $i$. We then consider the interval $\hat{J}=[a:b]\subseteq [1:n]$ containing $i$, defined as the \textit{largest subinterval on which} $\hat{\theta}_i$ \textit{is a local maximum}. Now, for any interval $I\subseteq \hat{J}$ containing $i$, we show that if one constructs a new vector by applying a constant positive perturbation to $\hat{\theta}^{\max}$ at each location in $I$, then the corresponding objective function varies affinely with the magnitude of the perturbation. Exploiting this affine structure via a directional derivative argument, we further show that if $\hat{\theta}^{\max}_i<U_i$, then one can increase $\hat{\theta}^{\max}$ on $I$ by a small positive amount without increasing the value of the objective function. Since $I$ contains $i$, this contradicts the maximality of $\hat{\theta}^{\max}_i$ at location $i$. Thus, we conclude that $\hat{\theta}^{\max}_i=U_i$. By a similar argument, we conclude that $\hat{\theta}^{\min}_i=L_i$, where $\hat{\theta}^{\min}\in \cS$ is such that $\hat{\theta}^{\min}_i$ attains the minimum possible value at location $i$. This perturbation argument is one of the primary technical innovations in the proof, and was not needed in the mean regression setting of \cite{chatterjee2026minmaxtrendfilteringgeneralizations}; see Proposition~\ref{lem:perturb-upper}, Lemma~\ref{lem:J_hat_in_tilde_I}, and their proofs. It therefore follows that $\cS_i=[L_i,U_i]$.
     \item Using the submodularity of the TV penalty together with the piecewise linear geometry of the quantile loss, we first establish that the solution set $\cS$ is closed under coordinatewise maximum and minimum. Then, by considering $n$ (possibly distinct) elements of $\cS$, each attaining the maximum (or minimum) possible value at a given coordinate, and taking their coordinatewise maximum (or minimum, respectively), we establish the existence of $\hat{\theta}^{Upper}$ (or $\hat{\theta}^{Lower}$, respectively).

     \item If there indeed existed crossing quantile TVD estimators for two different quantile levels $\tau_1<\tau_2$ with a common tuning parameter, then one can take coordinatewise maximum and minimum of those two estimators and using submodularity, we show that at least one of these two newly constructed estimators would yield strictly smaller value of the corresponding objective function in comparison to the initial crossing quantile TVD solutions. This violates the optimality of the latter and thus proves that crossing quantile TVD estimators with same tuning parameter can not exist.    
\end{enumerate}

 \begin{remark}
        The statement of Theorem~\ref{thm:minmax_bound} continues to hold even for quantile levels $\tau=0,1$ in a certain sense. Note that for $\tau=0$, it follows from Lemma~\ref{lem:minmax_finite} that $L_i=-\infty$ and $U_i=\min_{i\in [n]}y_i$. However, it is not difficult to see $\hat{\theta}$ is a quantile TVD estimator for $\tau=0$ if and only if $\hat{\theta}=(c,\ldots,c)^{\top}\in \bR^n$, where $c\leq \min_{i\in [n]}y_i$. Thus, Theorem~\ref{thm:minmax_bound} holds true for $\tau=0$ if we allow a convention that a constant vector with all entries $-\infty$ is a potential solution. A similar argument can be provided for $\tau=1$.
    \end{remark}

\subsection{Proof of Part $(1)$}\label{sec:minmax_1}
\noindent  We first note that when $\tau \in (0,1)$, the envelope bounds $L_i$ and $U_i$ are real numbers by Lemma~\ref{lem:minmax_finite}, thus $[L_i,U_i]$ is indeed a compact interval in $\bR$. We now divide the proof of part $(1)$ of Theorem~\ref{thm:minmax_bound} into the following three steps.
\begin{list}{\textbf{Step \arabic{enumi}:}}{
    \usecounter{enumi}
    \setlength{\leftmargin}{4em}   
    \setlength{\labelwidth}{3.5em} 
    \setlength{\labelsep}{0.5em}   
}
\item \label{itm:step_1} The set $\cS_i$ is a compact interval in $\bR$.
    \item \label{itm:step_2} If $\theta_i \in \cS_i$, then $L_i \leq \theta_i \leq U_i$.
    \item \label{itm:step_3} There exists a solution $\hat{\theta}^{\max}$ (or $\hat{\theta}^{\min}$, respectively) such that $\hat{\theta}^{\max}_i \geq U_i$ (or $\hat{\theta}^{\min}_i \leq L_i$, respectively). 
\end{list}

Step~\ref{itm:step_2} implies that $\cS_i\subseteq [L_i,U_i]$. Moreover, under the light of Step~\ref{itm:step_2}, Step~\ref{itm:step_3} implies that $\hat{\theta}^{\max}_i= U_i$ (or $\hat{\theta}^{\min}_i= L_i$, respectively). Since $\cS_i$ is an interval by Step~\ref{itm:step_1}, this means for any value $\theta$ in the interval $[L_i,U_i]$, there is a corresponding quantile TVD estimate whose point estimate at the $i$-th location is $\theta$, i.e., $[L_i,U_i]\subseteq \cS_i$.

\begin{proof}[\underline{Proof of Step~\ref{itm:step_1}:}]
For the sake of completeness, let us first argue that the set $\cS$ is non-empty. We define $F$ to be the objective function of interest, i.e., 
\begin{equation}\label{eqn:objective_fn}
    F^{\tau}(\theta):=\sum_{i=1}^{n}\rho_{\tau}(y_i-\theta_i)+\lambda\,TV(\theta),\qquad \theta \in \bR^n.
\end{equation}
We will write $F$ instead whenever the context is clear. $F$ is non-negative, continuous and convex. Therefore,  $0\leq I:=\inf_{\theta \in \bR^n} F(\theta)< \infty$. Moreover, observe that $F(\theta)\geq\sum_{i=1}^{n}\rho_{\tau}(y_i-\theta_i)$ and the RHS goes to infinity whenever $\left\|\theta\right\|\to \infty$. Hence, there is a compact rectangle $\mathcal{R}\in \bR^n$ such that $F(\theta)\geq I+1$ whenever $\theta \notin \mathcal{R}$. Moreover, $F$ being continuous in the compact set $\mathcal{R}$, the global infimum of $F$ is attained at some point in $\mathcal{R}$. Thus, $\cS$ is non-empty and bounded.
Now, note that convexity of $F$ implies that for any $a,b\in \cS$ and $\gamma\in [0,1]$,
\[I\leq F(\gamma\,a+(1-\gamma)\,b)\leq \gamma F(a)+(1-\gamma)F(b)=I.\]
The above implies that both $\cS$ is a convex set and its projection $\cS_i$ is also a convex set. Since $\cS_i\subseteq \bR$, convexity implies $\cS_i$ must be an interval. It remains to show the interval is closed. However, continuity of $F$ implies that its pre-image of the closed $\{I\}$, which is actually $\cS$, is also a closed set. Thus, Step~\ref{itm:step_1} is proved.
\end{proof}

\begin{proof}[\underline{Proof of Step~\ref{itm:step_2}:}]
The key technical ingredient of the proof of Step~\ref{itm:step_2} is the following subgradient characterization of any solution to the objective function in \eqref{eqn:defn_qtl_TVD}.
\begin{lemma}\label{lem:qtvd-interval}
$\hat{\theta}\in \cS$ if and only if there exist vectors $g=(g_1,\dots,g_n)\in\R^n$ and $z=(z_0,\dots,z_n)\in\R^{n+1}$ such that
\begin{enumerate}
\item $z_0=z_n=0$, and for each $k=1,\dots,n-1$,
\[
z_k\in
\begin{cases}
\{\lambda\}, & \hat\theta_k>\hat\theta_{k+1},\\
[-\lambda,\lambda], & \hat\theta_k=\hat\theta_{k+1},\\
\{-\lambda\}, & \hat\theta_k<\hat\theta_{k+1};
\end{cases}
\]
\item for each $j$,
\[
g_j \in
\partial_{\theta_j}\rho_\tau(y_j-\hat\theta_j)=
\begin{cases}
\{-\tau\}, & \hat\theta_j<y_j,\\
[-\tau,1-\tau], & \hat\theta_j=y_j,\\
\{1-\tau\}, & \hat\theta_j>y_j;
\end{cases}
\]
\item for every interval $I=[a:b]\subseteq[n]$,
\begin{equation}\label{eq:interval-id-qtvd}
\sum_{j=a}^b g_j = z_{a-1}-z_b,
\end{equation}
\end{enumerate}
where $\partial_{\theta_j} \rho_{\tau}(y_j-\hat{\theta}_j)$ is the subgradient of the function $\rho_{\tau}(y_j-\theta_j)$ with respect to $\theta_j$, evaluated at $\hat{\theta}_j$.
\end{lemma}

\begin{proof}[Proof of Lemma~\ref{lem:qtvd-interval}]
Let $D:\R^n\to\R^{n-1}$ be the matrix such that $(D\theta)_k=\theta_{k+1}-\theta_k$. We write the objective function in \eqref{eqn:defn_qtl_TVD} as
$F(\theta)=G(\theta)+\lambda\|D\theta\|_1$ with
$G(\theta)=\sum_{j=1}^n\rho_\tau(y_j-\theta_j)$. Since $F$ is convex, KKT condition implies that $\hat{\theta}\in \cS$ if and only if $0\in\partial F(\hat\theta)$. The latter is equivalent to saying that there exist
$g\in\partial G(\hat\theta)$ and $s\in\partial\|D\hat\theta\|_1$ with $0=g+\lambda D^\top s$.
Componentwise,
\[
s_k\in
\begin{cases}
\{+1\}, & \hat\theta_{k+1}>\hat\theta_k,\\
[-1,1], & \hat\theta_{k+1}=\hat\theta_k,\\
\{-1\}, & \hat\theta_{k+1}<\hat\theta_k.
\end{cases}
\]
Define
\[
z_0=0,\qquad z_k=-\lambda s_k\ (k=1,\dots,n-1),\qquad z_n=0.
\]
Then the stated bounds on $z_k$ hold. Moreover a direct computation gives
\[
\lambda D^\top s=(z_1-z_0,\ z_2-z_1,\ \dots,\ z_n-z_{n-1}),
\]
so $g_j=z_{j-1}-z_j$ for each $j$, and summing over $j=a,\dots,b$ telescopes to \eqref{eq:interval-id-qtvd}.
\end{proof}
 Now, pick an arbitrary $\theta\in \cS_i$ and a corresponding $\hat{\theta}\in \cS$ such that $\hat{\theta}_i=\theta$. We would first prove the upper bound in Step~\ref{itm:step_2} and then argue why that automatically implies the lower bound. In order to establish the upper bound, it suffices to show that for any $J\in \cI$ containing $i$, it holds that
 \[\hat{\theta}_i\leq \max_{I\in \I: I\subseteq J,\, i \in I}\, y_{I,(\lfloor u_{I,J}\rfloor+1)}.\]
 To this end, fix an outer interval $J=[a:b]$ containing $i$.
Let $I=[c:d]\subseteq J$ be the \emph{largest} subinterval of $J$ containing $i$ such that
\[
\hat\theta_u\ge \hat\theta_i\qquad \forall u\in I.
\]
Then $\hat\theta_j\ge \hat\theta_i$ for all $j\in I$, hence
\[
\#\{j\in I: y_j<\hat\theta_i\} \le \#\{j\in I:\hat\theta_j>y_j\}.
\]
Now, note that it follows from part $(ii)$ of Lemma~\ref{lem:qtvd-interval} that
\begin{equation}\label{eq:gj-onesided}
g_j \ge \one\{\hat\theta_j>y_j\}-\tau,
\qquad
g_j \le \one\{\hat\theta_j\ge y_j\}-\tau.
\end{equation}
Summing the first inequality in \eqref{eq:gj-onesided} over $j\in I$ gives
\[
\sum_{j=c}^d g_j \ge \#\{j\in I:\hat\theta_j>y_j\}-\tau|I|.
\]
Combining with part $(iii)$ of Lemma~\ref{lem:qtvd-interval} yields
\begin{equation*}
0\le \#\{j\in I:\hat\theta_j>y_j\}\le \tau|I|+(z_{c-1}-z_d).
\end{equation*}
Since $\hat{\theta}_j\geq \hat{\theta}_i$ for any $j \in I$, it follows that
\begin{equation}\label{eq:count-upper-thm1}
    0\le \#\{j\in I:\hat\theta_i>y_j\}\le \lfloor \tau|I|+(z_{c-1}-z_d)\rfloor,
\end{equation}
where the second inequality holds because \(\#\{j\in I:\hat\theta_i>y_j\}\) is an integer. Through the following  lemma, we first argue why~\eqref{eq:count-upper-thm1} implies an upper bound on $\hat{\theta}_i$ in terms of order statistics of the data restricted to the interval $I$.
\begin{lemma}\label{lem:orderstat}
Let $I\subseteq[n]$ and $t\in\R$. Then, for any integer $m$,
\begin{enumerate}
\item if $\#\{j\in I: t>y_j\}\le m$, then $t\le y_{I,(m+1)}$;
\item if $\#\{j\in I: t\geq y_j\}\ge m$, then $t\ge y_{I,(m)}$,
\end{enumerate}
where $y_{I,(k)}$ is defined in Definition~\ref{defn:Y_I,k}.
\end{lemma}
Thus, under the light of Lemma~\ref{lem:orderstat}, \eqref{eq:count-upper-thm1} implies
\begin{equation}\label{eq:count-upper-thm1-part2}
    \hat{\theta}_i\leq y_{I,\left(\lfloor \tau|I|+(z_{c-1}-z_d)\rfloor+1\right)}.
\end{equation}
An upper bound on $z_{c-1}-z_d$ is now required. Note that a naive bound would be $2\lambda$ as $|z_k|\leq \lambda$ for all $k$. However, it turns out that this bound could be loose for certain configurations of $I$ and $J$. For example, if $J=[1:n]$ and $I=[c:d]$ where $1<c\leq d=n$, then maximality of $I$ forces $z_{c-1}=-\lambda$ (see Lemma~\ref{lem:qtvd-interval}) and hence $z_{c-1}-z_d=z_{c-1}=-\lambda$. This is exactly where the constants $C_{I,J}$ enter the proof. The following lemma provides upper bound on the quantity $z_{c-1}-z_d$ for the specific choice of $I$. 
\begin{lemma}\label{lemma:z_upper_bd}
    Consider any location $i\in [n]$ and fix an arbitrary minimizer $\hat\theta\in\cS$ and an outer interval $J=[a:b]$ containing $i$. Let $I=[c:d]\subseteq J$ be the largest subinterval of $J$ containing $i$ such that
\[
\hat\theta_u\ge \hat\theta_i\qquad \forall u\in I.
\]
Then, it holds that
\begin{equation}\label{eq:z-upper-key-thm1}
z_{c-1}-z_d \le -2\lambda\,C_{I,J}.
\end{equation}
where $z_0,z_1,\cdots,z_n$ are the corresponding dual variables defined in Lemma~\ref{lem:qtvd-interval}.
\end{lemma}
\begin{remark}
    The proof of Lemma~\ref{lemma:z_upper_bd} actually can be found inside the proof of Theorem 1.1 in \cite{chatterjee2026minmaxtrendfilteringgeneralizations}. However, we still include the proof in Appendix~\ref{appna} for self-containment. It basically does a careful bookkeeping of all possible cases whether one or both end points of $I$ coincide(s) with that of $J$ and then uses the appropriate characterization of $z_{c-1}$ or $z_d$ via part $(1)$ of Lemma~\ref{lem:qtvd-interval}.
\end{remark}
We now come back to the original proof. Using Lemma~\ref{lemma:z_upper_bd}, it follows from~\eqref{eq:count-upper-thm1-part2} that
\[\hat{\theta}_i\leq y_{I,(\lfloor u_{I,J}\rfloor+1)}\leq \max_{I\in \cI:\,I\subseteq J,\,i\in I}\,y_{I,(\lfloor u_{I,J}\rfloor+1)},\]
where $u_{I,J}=\tau|I|-2\lambda\,C_{I,J}$. Since the above inequality holds for any arbitrary $J$ containing $i$, 
the upper bound is thus established.\\

We now argue why this implies $L_i\leq \hat{\theta}_i$. First, we observe that establishing the lower bound is equivalent to proving $-\hat{\theta}_i\leq -L_i$. Next, we also observe that from the definition of the quantile loss, it follows that $-\hat{\theta}$ is a quantile TVD estimate to the data $-y$ with the same tuning parameter $\lambda\geq 0$, but for quantile level $1-\tau$, i.e.,
\[-\hat{\theta}\in \argmin_{\theta \in \bR^n}\sum_{i=1}^{n}\rho_{1-\tau}(y_i-\theta_i)+\lambda\,TV(\theta).\]
Therefore, the already established upper bound implies
\[-\hat{\theta}_i\leq \min_{J\in \cI:i\in J}\,\max_{I\in \cI:\,I\subseteq J,\,i\in I}\,(-y)_{I,(\lfloor (1-\tau)|I|-2\lambda C_{I,J}\rfloor+1)}.\]
Now, observe the following identity that relates the order statistics of a vector to those of the negative of the vector itself:
\begin{equation}\label{eqn:order_stat_neg_vector}
    (-y)_{I,(|I|-k+1)}=-(y_{I,(k)}),
\end{equation}
where we follow the general convention on order statistics defined in Definition~\ref{defn:Y_I,k}. Therefore,
\begin{align*}
    -\hat{\theta}_i&\leq \min_{J\in \cI:i\in J}\,\max_{I\in \cI:\,I\subseteq J,\,i\in I}\,(-y)_{I,(\lfloor (1-\tau)|I|-2\lambda C_{I,J}\rfloor+1)}\\
    &=\min_{J\in \cI:i\in J}\,\max_{I\in \cI:\,I\subseteq J,\,i\in I}\,-(y_{I,(|I|-\lfloor (1-\tau)|I|-2\lambda C_{I,J}\rfloor)}).
\end{align*}
Now, observe that  for any real number $a$, it holds that \(|I|-\lfloor a \rfloor\geq \lceil |I|-a\rceil.\) This is true because if $k\leq a < k+1$, where $k$ is an integer, then $|I|-k-1<|I|-a\leq |I|-k=|I|-\lfloor a \rfloor$. Therefore, we can bound \(|I|-\lfloor (1-\tau)|I|-2\lambda C_{I,J}\rfloor\) below by \[\lceil |I|-(1-\tau)|I|-2\lambda C_{I,J}\rceil=\lceil \tau |I|+2\lambda C_{I,J}\rceil=\lceil l_{I,J}\rceil.\]
Therefore,
\begin{align*}
    -\hat{\theta}_i&\leq \min_{J\in \cI:i\in J}\,\max_{I\in \cI:\,I\subseteq J,\,i\in I}\,-(y_{I,(\lceil l_{I,J}\rceil)})\\
    &=-\max_{J\in \cI:i\in J}\,\min_{I\in \cI:\,I\subseteq J,\,i\in I}\,y_{I,(\lceil l_{I,J}\rceil)}\\
    &=-L_i.
\end{align*}
Thus, the lower bound is established and Step~\ref{itm:step_2} is proved.
\end{proof}

\begin{proof}[\underline{Proof of Step~\ref{itm:step_3}:}]
In this step, we would like to establish that there exists a $\hat{\theta}\in \cS$ such that $\hat{\theta}_i\geq U_i$. Consider any arbitrary $\hat{\theta}\in \cS$. It suffices to show that there exists an interval $\hat{J}$ (depends on $\hat{\theta}$) containing $i$ such that
    \[\hat\theta_i\geq \max_{\substack{I\in I: I\subseteq \hat{J},\; i\in I}} y_{I,(\lfloor u_{I,\hat{J}}\rfloor+1)},\]
    i.e.,
    \begin{equation}\label{eqn:req_lower_bd_lower_bd}
       \hat\theta_i\geq  y_{I,(\lfloor u_{I,\hat{J}}\rfloor+1)},\text{ for all }I\subseteq \hat{J}. 
    \end{equation}
    To this end, let us choose $\hat{J}=[a:b]$ to be the \emph{largest} interval containing $i$ such that
\[
\hat\theta_u\le \hat\theta_i,\qquad \forall u\in \hat{J}.
\]
We again turn back to the subgradient characterization of $\hat{\theta}$ given by Lemma~\ref{lem:qtvd-interval}. Fix an arbitrary subinterval $I=[c:d]\subseteq \hat{J}$. Summing the second inequality in \eqref{eq:gj-onesided} over $j\in I$ and using part $(iii)$ of Lemma~\ref{lem:qtvd-interval} gives
\begin{equation*}
z_{c-1}-z_d=\sum_{j=c}^d g_j \le \#\{j\in I:\hat\theta_j\ge y_j\}-\tau|I|.
\end{equation*}
By definition of $\hat{J}$ and using the fact that $\#\{j\in I:\hat\theta_i\ge y_j\}$ is a non-negative integer, we thus have
\begin{equation}
    \lceil z_{c-1}-z_d+\tau |I|\rceil  \le \#\{j\in I:\hat\theta_j\ge y_j\}\le \#\{j\in I:\hat\theta_i\ge y_j\}.
\end{equation}
Therefore, it follows from Lemma~\ref{lem:orderstat} that
\begin{equation}\label{eq:count-lower-start}
    \hat{\theta}_i\geq y_{I,(\lceil z_{c-1}-z_d+\tau |I|\rceil)}.
\end{equation}
The following lemma provides a lower bound on $z_{c-1}-z_d$ in terms of $\lambda$ and $C_{I,\hat{J}}$ by utilizing part $(1)$ of Lemma~\ref{lem:qtvd-interval} and the definition of $\hat{J}$.
\begin{lemma}\label{lemma:z_lower_bd}
    Consider any $\hat{\theta}\in \cS$ and fix any location $i \in [n]$. Let $\hat{J}=[a:b]$ be the \emph{largest} interval containing $i$ such that
\[
\hat\theta_u\le \hat\theta_i,\qquad \forall u\in \hat{J}.
\]
Then, for any $I=[c,d]\subseteq \hat{J}$, the following inequality holds: 
\begin{equation}\label{eq:z-lower-key-thm2}
z_{c-1}-z_d \ge -2\lambda\,C_{I,\hat{J}}.
\end{equation}
\end{lemma}
    \begin{remark}
    The proof of Lemma~\ref{lemma:z_lower_bd} can be found in \cite{chatterjee2026minmaxtrendfilteringgeneralizations}. However, we still include the proof in Appendix~\ref{appna} for self-containment. 
\end{remark}

Plugging this lower bound in \eqref{eq:count-lower-start}, we obtain
\begin{equation}\label{eq:u_IJhat_upper}
    \hat{\theta}_i\geq y_{I,(\lceil u_{I,\hat{J}}\rceil)}.
\end{equation}
So far in the proof, we have followed an extension of the roadmap laid out in \cite{chatterjee2026minmaxtrendfilteringgeneralizations} that uses interval identities for the two kinds of maximal intervals. However, notice that while the lower bound in \eqref{eq:u_IJhat_upper} holds for any quantile TVD solution, it is not our desired bound~\eqref{eqn:req_lower_bd_lower_bd} and in fact, is a weaker bound. For example, when $0\le u_{I,\hat{J}}<|I|$ and $u_{I,\hat J}\in\mathbb Z$, \(\lceil u_{I,\hat{J}}\rceil<\lfloor u_{I,\hat{J}}\rfloor+1,\)
and thus, the bound in~\eqref{eq:u_IJhat_upper} seems to be off by one order statistic from the required bound in~\eqref{eqn:req_lower_bd_lower_bd}. This is where our proof needs a substantially different ingredient from \cite{chatterjee2026minmaxtrendfilteringgeneralizations}. In particular, we consider an element $\hat{\theta}^{\max} \in \cS$ such that
\[\hat{\theta}^{\max}_i=\max\{\theta_i:\;\theta \in \cS\}.\]
Note that in Step~\ref{itm:step_1}, we have already proved that $\cS_i$ is compact and hence, the maximum above is well-defined. Now, we develop a new perturbation argument and crucially exploit the maximality of $\hat{\theta}^{\max}$ to conclude that \eqref{eq:u_IJhat_upper} indeed implies \eqref{eqn:req_lower_bd_lower_bd} for $\hat{\theta}^{\max}$.
 We consider the following four possible cases.
\begin{itemize}
\item \underline{\textbf{Case A: $u_{I,\hat{J}}<0$.}} In this case, \(y_{I,(\lfloor u_{I,\hat{J}}\rfloor+1)}=-\infty\) and thus, 
\[\hat\theta^{\max}_i\ge y_{I,(\lfloor u_{I,\hat{J}}\rfloor+1)}.\]

\item \underline{\textbf{Case B: $0\le u_{I,\hat{J}}<|I|$ and $u_{I,\hat{J}}\notin\mathbb Z$.}}
In this case,
\[
\lceil u_{I,\hat{J}}\rceil=\lfloor u_{I,\hat{J}}\rfloor+1.
\]
Plugging this into~\eqref{eq:u_IJhat_upper}, we get
\[
\hat\theta_i^{\max}\ge y_{I,(\lfloor u_{I,\hat{J}}\rfloor+1)}.
\]

\item \underline{\textbf{Case C: $0\le u_{I,\hat{J}}<|I|$ and $u_{I,\hat J}\in\mathbb Z$.}}
Note that in this case,
\[\lceil u_{I,\hat{J}}\rceil<\lfloor u_{I,\hat{J}}\rfloor+1,\]
and thus, the bound in~\eqref{eq:u_IJhat_upper} seems to be off by one order statistic from the required bound in~\eqref{eqn:req_lower_bd_lower_bd}. However, we now introduce a novel perturbation argument and crucially utilize the maximality of $\hat{\theta}^{\max}$ at location $i$ to establish that for our specific choice of $\hat{\theta}^{\max}$ and $\hat{J}$, \eqref{eq:u_IJhat_upper} indeed implies the required lower bound.
\begin{proposition}[One order statistic upgrade]\label{lem:perturb-upper}
 Fix any location $i \in [n]$ and define
    \[\hat\theta_i^{\max}:=\max\{\theta_i:\theta\in\mathcal \cS\}.\]
    Moreover, let $\hat{\theta}^{\max}$ be an element in $\cS$ corresponding to $\hat{\theta}^{\max}_i$.
Let $\hat{J}=[a:b]$ be the largest interval containing $i$ such that
\[
\hat\theta^{\max}_u \le \hat\theta^{\max}_i
\qquad \forall u\in \hat J.
\]
Fix any $I=[c:d]\subseteq \hat J$ with $i\in I$ and suppose
$u_{I,\hat J}\in\mathbb Z$ and $0\le u_{I,\hat J}<|I|$.
Then,
\[
\hat\theta_i^{\max}\ge y_{I,(u_{I,\hat J}+1)}.
\]
\end{proposition}
\begin{proof}[Proof of Proposition~\ref{lem:perturb-upper}]
Let $K:=u_{I,\hat J}$ and set
\[
q:=y_{I,(K+1)}.
\]
Suppose for contradiction that $\hat\theta_i^{\max}<q$.
Define
\begin{equation}\label{defn:theta_t}
    \theta(t):=\hat\theta^{\max}+t\,\one_I,\qquad t\ge 0,
\end{equation}

and $\phi(t):=F(\theta(t))$. We will argue that $\phi$ is affine in some $[0,\eta_1]$ for some small enough $\eta_1>0$ and $\phi'_{+}(0)\leq 0$. To that end, for $j\in I$ write
\[
u_j(t):=y_j-\theta_j(t)=y_j-\hat\theta^{\max}_j-t.
\]
If $y_j>\hat\theta^{\max}_j$, then $u_j(0)>0$ and the sign remains positive for
$t<y_j-\hat\theta^{\max}_j$.
If $y_j=\hat\theta^{\max}_j$, then $u_j(t)=-t<0$ for all $t>0$.
If $y_j<\hat\theta^{\max}_j$, then $u_j(t)<0$ for all $t\ge 0$.
Hence, letting
\[
\eta:=\min_{j\in I:\,y_j>\hat\theta^{\max}_j}
(y_j-\hat\theta^{\max}_j)>0,
\]
(with $\eta=+\infty$ if the index set is empty), we conclude that $u_j(t)$ does not change sign for $t\in [0,\eta]$ for any $j \in I.$ Thus, from the definition of $\rho_{\tau}$, it follows that each $\rho_\tau(u_j(t))$ is affine on $[0,\eta]$,
and so is their sum. Thus,
\[
\frac{d}{dt}\rho_\tau(u_j(t))\Big|_{0^+}
=
\begin{cases}
-\tau, & \hat\theta^{\max}_j<y_j,\\
1-\tau, & \hat\theta^{\max}_j\ge y_j.
\end{cases}
\]
Summing gives
\[
\frac{d}{dt}\sum_{j=c}^d \rho_\tau(y_j-\theta_j(t))\Big|_{0^+}
=
-\tau|I|+\#\{j\in I:\hat\theta^{\max}_j\ge y_j\}.
\]

Now, we turn our attention to the penalty term. Only the two boundary edges across $\partial \hat{J}$ depend on $t$ in the TV term. In the following lemma, we establish that $TV(\theta(t))$ is affine in a neighborhood around $0$.
\begin{lemma}\label{lemma:TV_is_affine}
    Let $\theta(t)$ be as defined in~\eqref{defn:theta_t}. Then, for small enough $t>0$, it holds that
    \begin{equation}\label{eqn:penalty_term_is_affine}
    TV(\theta(t))-TV(\theta(0))=(\sigma_{c-1}\one\{c\neq 1\}-\sigma_d\one\{d\neq n\})t,
\end{equation}
where for $c>1$ and $d<n$,
\[
\sigma_{c-1}:=\operatorname{sign}_+\bigl(\hat\theta^{\max}_c-\hat\theta^{\max}_{c-1}\bigr),
\qquad
\sigma_d:=\operatorname{sign}_-\bigl(\hat\theta^{\max}_{d+1}-\hat\theta^{\max}_d\bigr),
\]
and $\operatorname{sign}_+,\;\operatorname{sign}_-$ are defined as
\[
\operatorname{sign}_+(x):=
\begin{cases}
+1,& x\ge 0,\\
-1,& x<0,
\end{cases}
\qquad
\operatorname{sign}_-(x):=
\begin{cases}
+1,& x>0,\\
-1,& x\le 0.
\end{cases}
\]
\end{lemma}
Hence, we conclude that
\[\frac{d}{dt}\Big(\lambda\sum_{k}|\theta_{k+1}(t)-\theta_k(t)|\Big)\Big|_{0^+}
=
\lambda(\sigma_{c-1}\one\{c\neq 1\}-\sigma_d\one\{d\neq n\}).\]
Thus, we conclude that there exists $0<\eta_1\leq \eta$ such that $\phi$ is affine in $[0,\eta_1]$ and 
\begin{equation}\label{eqn:phi_deriv}
    \phi'_+(0)=-\tau|I|+\#\{j\in I:\hat\theta^{\max}_j\ge y_j\}+\lambda(\sigma_{c-1}\one\{c\neq 1\}-\sigma_d\one\{d\neq n\}).
\end{equation}
Now, since $I\subseteq J$ and $\hat\theta^{\max}_j\le\hat\theta^{\max}_i$ for all $j\in I$,
\[
\#\{j\in I:\hat\theta^{\max}_j\ge y_j\}
\le
\#\{j\in I:y_j\le \hat\theta_i^{\max}\}.
\]
However, because we assumed $\hat\theta_i^{\max}<q=y_{I,(K+1)}$,
\[
\#\{j\in I:y_j\le \hat\theta_i^{\max}\}\le \#\{j\in I:y_j< y_{I,(K+1)}\}= K.
\]
Thus,
\begin{equation}\label{eq:data-bound-full}
-\tau|I|
+\#\{j\in I:\hat\theta^{\max}_j\ge y_j\}
\leq
-\tau|I|+K
=
-2\lambda C_{I,\hat{J}},
\end{equation}
where the last equality follows from the definition of $u_{I,\hat{J}}=K.$
Now, in the following lemma, we establish an upper bound on the third term of~\eqref{eqn:phi_deriv}.
\begin{lemma}\label{lemma:tv-bound-full-lemma}
    \begin{equation}\label{eq:tv-bound-full-lemma}
\lambda(\sigma_{c-1}\one\{c\neq 1\}-\sigma_d\one\{d\neq n\})\le 2\lambda C_{I,J},
\end{equation}
where the associated quantities are defined in Lemma~\ref{lemma:TV_is_affine}.
\end{lemma}
Plugging in \eqref{eq:data-bound-full} and \eqref{eq:tv-bound-full-lemma} in \eqref{eqn:phi_deriv}, we conclude that $\phi'_+(0)\le 0$. However, since $\hat\theta^{\max}$ minimizes $F$, $\phi'_+(0)\ge 0$. Thus, it must holds that $\phi'_+(0)=0$.
However, we have already argued that $\phi$ is affine on $[0,\eta_1]$. Therefore, it must hold that $\phi(t)=\phi(0)$ for $t\in[0,\eta_1]$. Choose $\varepsilon\in(0,\eta_1]$ and define
$\tilde\theta=\hat\theta^{\max}+\varepsilon\one_I$.
Then $\tilde\theta\in\mathcal \cS$ and
$\tilde\theta_i>\hat\theta_i^{\max}$,
contradiction! Thus, the assumption that $\hat\theta_i^{\max}<q$ must be wrong and thus, the proof of Proposition~\ref{lem:perturb-upper} is complete.
\end{proof}
\item \underline{\textbf{Case D: $ u_{I,\hat{J}}\geq |I|$.}}
In the following lemma, we show that for our choice of $\hat{\theta}^{\max}$ and $\hat{J}$, it holds that $u_{I,\hat{J}}<|I|$ for all $I\subseteq \hat{J}$ containing $i$. Thus, case D never arises.
\begin{lemma}\label{lem:J_hat_in_tilde_I}
    Fix any location $i \in [n]$ and define
    \[\hat\theta_i^{\max}:=\max\{\theta_i:\theta\in\mathcal \cS\}.\]
    Moreover, let $\hat{\theta}^{\max}$ be an element in $\cS$ corresponding to $\hat{\theta}^{\max}_i$ and let $\hat{J}=[a:b]$ be the \emph{largest} interval containing $i$ such that
\[
\hat\theta^{\max}_u\le \hat\theta^{\max}_i,\qquad \forall u\in \hat{J}.
\]
Then, it holds that $u_{I,\hat{J}}<|I|$ for all $I\subseteq \hat{J}$ such that $i\in I$.

\end{lemma}
The proof uses a perturbation argument very similar to Proposition~\ref{lem:perturb-upper} and is deferred to Appendix~\ref{appna}.

\end{itemize}
Combining the four cases, we conclude the proof of~\ref{eqn:req_lower_bd_lower_bd}. Thus, $\hat{\theta}^{\max}_i\geq U_i$. Now, along the same lines of the proof of Step~\ref{itm:step_2}, where we showed why the upper bound $U_i$ on $\hat{\theta}_i$ implies the lower bound of $L_i$, one can similarly conclude that \(\hat{\theta}^{\min}_i\leq L_i,\) where $\hat{\theta}^{\min}$ is an element of $\cS$ such that
\[\hat{\theta}^{\min}_i=\min\{\theta:\;\theta \in \cS\}.\]
Thus, proof of Step~\ref{itm:step_3} and therefore, part $(1)$ of Theorem~\ref{thm:minmax_bound} is now proved.
\end{proof}

\subsection{Proof of Part $(2)$}\label{sec:minmax_2}
By part $(1)$, for each $i\in [n]$, there exists $\hat{\theta}^{(i)}\in \cS$ such that \(\hat{\theta}^{(i)}_i=U_i.\) Now, define $\tilde{\theta}^{(1)}:=\hat{\theta}^{(1)}$ and for $2\leq k \leq n$, define $\tilde{\theta}^{(k)}:=\tilde{\theta}^{(k-1)}\vee \hat{\theta}^{(k)},$ where $\vee$ denotes coordinatewise maximum. Observe that since we know from part $(1)$ that $\hat{\theta}^{(i)}_j\leq U_j$ for any $i,j\in [n]$, thus we can easily conclude that $\tilde{\theta}^{(n)}_i=U_i$ for all $i\in [n]$. Now, the following lemma tells us that each $\tilde{\theta}^{(k)}$ is an element of $\cS$.
\begin{lemma}\label{lemma:minmax_preserves_qtltvd_tvdonly}
     Let $0<\tau_1 \leq \tau_2<1$ and let
\[
\hat{\theta}^{(i)} \in \cS^{\tau_i},\quad i=1,2.
\]
Define
\[
\tilde{\theta}^{(1)} = \hat{\theta}^{(1)} \wedge \hat{\theta}^{(2)},\quad\tilde{\theta}^{(2)} = \hat{\theta}^{(1)} \vee \hat{\theta}^{(2)},
\]
where $\vee$ and $\wedge$ denote coordinatewise maximum and minimum.

Then,
\[
\tilde{\theta}^{(i)} \in \cS^{\tau_i},\quad i=1,2.
\]
\end{lemma}

The above result is a consequence of the "submodularity" property of the TV penalty and piecewise linear geometry of the quantile loss function and its proof will be provided in Section~\ref{sec:monotonicity}. Now, getting back to the original proof, by a simple inductive argument, it follows  that $\tilde{\theta}^{(n)}\in \cS$. Thus, we have proved the existence of a solution $\tilde{\theta}^{(n)}\in \cS$ that achieves the upper envelope bound $U_i$ at all locations $i\in[n]$ simultaneously. The existence of a solution that achieves the lower envelope bound $L_i$ at all locations can be proved similarly by taking coordinatewise minimum sequentially.

\subsection{Proof of Part $(3)$}\label{sec:minmax_3}

Let $0<\tau_1<\tau_2<1$ and pick any $\hat{\theta}^{\tau_i}\in \cS^{\tau_i}$, $i=1,2$. For $i=1,2$ let $F^{\tau_i}$ be the corresponding objective functions defined in \eqref{eqn:objective_fn}. Now, define
\[
\tilde{\theta}^{(1)} = \hat{\theta}^{\tau_1} \wedge \hat{\theta}^{\tau_2},\quad\tilde{\theta}^{(2)} = \hat{\theta}^{\tau_1} \vee \hat{\theta}^{\tau_2},
\]
where $\vee$ and $\wedge$ denote coordinatewise maximum and minimum. Then, if $\hat{\theta}^{\tau_1}_i>\hat{\theta}^{\tau_2}_i$ at any location $i\in [n]$, the proof of Theorem~\ref{thm:submodular_non_crossing} reveals that submodularity of TV penalty and piecewise linear geometry of the quantile loss implies
\[F^{\tau_1}(\tilde{\theta}^{(1)})+F^{\tau_2}(\tilde{\theta}^{(2)})<F^{\tau_1}(\hat{\theta}^{\tau_1})+F^{\tau_2}(\hat{\theta}^{\tau_2}).\]
This contradicts the fact $\hat{\theta}^{\tau_i}\in \cS^{\tau_i}$ for $i=1,2$ and hence, we conclude that crossing can not happen. For explicit mathematical details, we refer the reader to Section~\ref{sec:monotonicity}.

\section{Monotonicity of Penalized Quantile Regression with Submodular Penalty}\label{sec:monotonicity}  
As we have already seen in Theorem~\ref{thm:minmax_bound} that the quantile TVD estimator at different quantile levels with a common tuning parameter can never produce crossing solutions. We will now discuss more details on the proof of this phenomenon and generalize this to submodular penalties.
\begin{definition}
    Let $P:\mathcal{X}\mapsto \bR$, where $\mathcal X=\prod_{i=1}^n \mathcal X_i\subseteq \bR^n$ and each $\mathcal X_i\subseteq \bR$. Then, $P$ is called \textit{submodular} if and only if for all $(\mathbf{x},\mathbf{y})\in \mathcal X \times \mathcal X,$
    \[P(\mathbf{x})+P(\mathbf{y})\geq P\left(\mathbf{x}\vee \mathbf{y}\right)+P\left(\mathbf{x} \wedge \mathbf{y}\right),\]
    where $\vee$ and $\wedge$ represents coordinatewise maximum and minimum of the two vectors, respectively.
\end{definition}
 
 Submodularity is a well-studied concept in combinatorial optimization and it has found numerous applications in machine learning, computer vision or signal
processing, see \cite{Bach2011LearningWS, bach2019submodular} for a detailed overview on this matter. A direct implication of the definition is that the set of minimizers of a submodular function is a lattice. Combining the notion of submodularity with the piecewise linear geometry of the quantile loss function, we show that any penalized quantile regression estimator with submodular penalty can never produce crossing solutions. We formally state the result below.

 \begin{theorem}\label{thm:submodular_non_crossing}
     Fix a data vector $y\in \bR^n$. For any $\tau \in (0,1)$ and $\lambda\geq 0$, define
     \[G^{\tau}(\theta)=\sum_{j=1}^n \rho_{\tau}(y_j-\theta_j)+\lambda\,P(\theta),\;\;\theta \in \bR^n.\]
     Consider any $0<\tau_1<\tau_2<1$ and let $\hat{\theta}^{\tau_i}$ be \textit{any} minimizer of $G^{\tau_i}$, for $i=1,2$, respectively. If the penalty $P$ is submodular, then it holds that
     \[\hat{\theta}^{\tau_1}_i\leq \hat{\theta}^{\tau_2}_i,\quad \forall\;i\in [n].\]
 \end{theorem}

The key technical ingredient for the proof of this theorem is the following lemma. It shows that for $0<\tau_1\leq \tau_2<1$ and any two minimizers of the corresponding objective functions of a penalized quantile regression estimator with a submodular penalty, the vector obtained from taking coordinatewise minimum (or maximum, respectively) remains minimizer of the objective function corresponding to $\tau_1$ (or \(\tau_2\), respectively).
 \begin{lemma}\label{lemma:minmax_preserves_qtltvd}
  Let $0<\tau_1 \leq \tau_2<1$ and let
\[
\hat{\theta}^{(i)} \in \argmin G^{\tau_i}(\theta),\quad i=1,2, 
\]
where
 \[G^{\tau}(\theta)=\sum_{j=1}^n \rho_{\tau}(y_j-\theta_j)+\lambda\,P(\theta),\;\;\theta \in \bR^n,\]
 and $P$ is submodular.
Define
\[
\tilde{\theta}^{(1)} = \hat{\theta}^{(1)} \wedge \hat{\theta}^{(2)},\quad\tilde{\theta}^{(2)} = \hat{\theta}^{(1)} \vee \hat{\theta}^{(2)},
\]
where $\vee$ and $\wedge$ denote coordinatewise maximum and minimum.

Then,
\[
\tilde{\theta}^{(i)} \in \argmin G^{\tau_i}(\theta),\quad i=1,2.
\]
 \end{lemma}
 \begin{proof}[Proof of Lemma~\ref{lemma:minmax_preserves_qtltvd}]
It follows from Lemma~\ref{lem:qtl_linearity} that the quantile loss is a linear function of the quantile level, or in other words, 
$$Q_{\tau_{2}}\left(\tilde{\theta}^{(2)}\right)=Q_{\tau_{1}}\left(\tilde{\theta}^{(2)}\right)+\left(\tau_{2}-\tau_{1}\right) \sum_{i=1}^{n}\left(y_{i}-\tilde{\theta}^{(2)}_{i}\right),$$
$$Q_{\tau_{2}}\left(\hat{\theta}^{(2)}\right)=Q_{\tau_{1}}\left(\hat{\theta}^{(2)}\right)+\left(\tau_{2}-\tau_{1}\right) \sum_{i=1}^{n}\left(y_{i}-\hat{\theta}_{i}^{(2)}\right),$$
where $Q_{\tau}(\theta)=\sum_{j=1}^n \rho_{\tau}(y_j-\theta_j)$.
Therefore, 
\begin{align}\label{eqn:qtl_loss_decomposition}
&Q_{\tau_{1}}\left(\hat{\theta}^{(1)}\right)+Q_{\tau_{2}}\left(\hat{\theta}^{(2)}\right) 
-  Q_{\tau_{1}}\left(\tilde{\theta}^{(1)}\right)-Q_{\tau_{2}}\left(\tilde{\theta}^{(2)}\right) \nonumber\\
 &= Q_{\tau_{1}}\left(\hat{\theta}^{(1)}\right)+Q_{\tau_{1}}\left(\hat{\theta}^{(2)}\right)+\left(\tau_{2}-\tau_{1}\right) \sum\left(y_{i}-\hat{\theta}_{i}^{(2)}\right) -Q_{\tau_{1}}\left(\tilde{\theta}^{(1)}\right)-Q_{\tau_{1}}\left(\tilde{\theta}^{(2)}\right) 
 -\left(\tau_{2}-\tau_{1}\right) \sum\left(y_{i}-\tilde{\theta}^{(2)}_i\right) \allowdisplaybreaks\nonumber\\
&= \left[\underbrace{Q_{\tau_{1}}\left(\hat{\theta}^{(1)}\right)+Q_{\tau_{1}}\left(\hat{\theta}^{(2)}\right)-Q_{\tau_{1}}\left(\tilde{\theta}^{(1)}\right)-Q_{\tau_{1}}\left(\tilde{\theta}^{(2)}\right)}_{=0}\right]  +\left(\tau_{2}-\tau_{1}\right) \sum_{i=1}^{n}\left(\tilde{\theta}^{(2)}_i-\hat{\theta}_{i}^{(2)}\right)\allowdisplaybreaks\nonumber\\
&=\left(\tau_{2}-\tau_{1}\right) \sum_{i=1}^{n}\left(\tilde{\theta}^{(2)}_i-\hat{\theta}_{i}^{(2)}\right)\geq 0,
\end{align}
where the last inequality follows from the fact $\tau_2>\tau_1$ and the definition of $\tilde{\theta}^{(2)}.$ Now, by submodularity,
\begin{equation}\label{eqn:submodular_lattice}
\lambda\,P(\tilde{\theta}^{(1)})+\lambda\,P(\tilde{\theta}^{(2)})\leq \lambda\,P(\hat{\theta}^{(1)})+\lambda\,P(\hat{\theta}^{(2)}).
\end{equation}
Therefore, we conclude that
\[G^{\tau_1}(\tilde{\theta}^{(1)})+G^{\tau_2}(\tilde{\theta}^{(2)})\leq G^{\tau_1}(\hat{\theta}^{(1)})+G^{\tau_2}(\hat{\theta}^{(2)}).\]
However, note that since $\hat{\theta}^{(i)}$s are minimizers of $G^{\tau_i}$s for $i=1,2$ (respectively), so must be $\tilde{\theta}^{(i)}$ (respectively). Thus, the proof of the lemma is complete. 
 \end{proof}

  We now proceed to prove Theorem~\ref{thm:submodular_non_crossing}.
\begin{proof}[Proof of Theorem~\ref{thm:submodular_non_crossing}]
     For $0<\tau_1<\tau_2<1$, let $\hat{\theta}^{\tau_i}$ be a minimizer of $G^{\tau_i}$, for $i=1,2$, respectively.     
     Let, if possible, there exist $j\in[n]$ such that $\hat{\theta}^{\tau_1}_j> \hat{\theta}^{\tau_2}_j$. Firstly, by Lemma~\ref{lemma:minmax_preserves_qtltvd}, $\tilde{\theta}^{(i)}$ minimizes $G^{\tau_i}$ for $i=1,2$, respectively, where
     \[
\tilde{\theta}^{(1)} = \hat{\theta}^{(1)} \wedge \hat{\theta}^{(2)},\quad\tilde{\theta}^{(2)} = \hat{\theta}^{(1)} \vee \hat{\theta}^{(2)},
\] 
Moreover, it is easy to see that $\tau_2>\tau_1$ and $\hat{\theta}^{\tau_1}_j> \hat{\theta}^{\tau_2}_j$ imply that the inequality in~\eqref{eqn:qtl_loss_decomposition} is strict, i.e.,
\begin{align*}
&Q_{\tau_{1}}\left(\hat{\theta}^{(1)}\right)+Q_{\tau_{2}}\left(\hat{\theta}^{(2)}\right) 
-  Q_{\tau_{1}}\left(\tilde{\theta}^{(1)}\right)-Q_{\tau_{2}}\left(\tilde{\theta}^{(2)}\right) >0.
\end{align*}
 Thus, in view of \eqref{eqn:submodular_lattice}, we conclude that
\[G^{\tau_1}(\tilde{\theta}^{(1)})+G^{\tau_2}(\tilde{\theta}^{(2)})< G^{\tau_1}(\hat{\theta}^{(1)})+G^{\tau_2}(\hat{\theta}^{(2)}).\]
However, since $\hat{\theta}^{(i)}$s are minimizers of $G^{\tau_i}$s for $i=1,2$, respectively, the above inequality can not be strict. Thus, we arrive at contradiction and conclude that there can not exist any $j\in [n]$ with $\hat{\theta}^{(1)}_j>\hat{\theta}^{(2)}_j$!
\end{proof}
 
In the following lemma, we characterize a general class of submodular penalties that encompasses the TV penalty.

 \begin{lemma}\label{lemma:general_submodular}
     Let $E\subseteq [n]\times [n]$. For each let $(i,j)\in E$ such that $i\neq j$, let $\phi_{ij}$ be a convex function on $\bR$ and $w_{ij}\geq 0$ be a corresponding weight. Then, the following function is submodular on $\bR^n$:
     \begin{equation}\label{eqn:example_submodular}
         P(\theta):=\sum_{(i,j)\in E:\;i\neq j}w_{ij}\,\phi_{ij}(\theta_i-\theta_j).
     \end{equation}
 \end{lemma}
 Although this is a well known fact in applied mathematics (see \cite{bach2019submodular}), we provide a proof of this result in Appendix~\ref{appnb} for the sake of completeness.
\begin{remark}
   Lemma~\ref{lemma:minmax_preserves_qtltvd_tvdonly} now directly follows by applying Lemma~\ref{lemma:minmax_preserves_qtltvd} on the TV penalty, which is submodular by Lemma~\ref{lemma:general_submodular}.
\end{remark}

\begin{remark}
Other examples of submodular penalties, as shown in Lemma~\ref{lemma:minmax_preserves_qtltvd_tvdonly}, include weighted sums of convex functions applied to the differences $\theta_{j}-\theta_i$. The choice of the convex function yields a range of important estimators: taking the absolute value recovers the weighted TVD on a general graph (see \cite{wang2016trend, ye2021non}, etc. for examples of such estimators), the squared loss leads to a discrete analogue of quantile smoothing spline of \cite{koenker1994quantile}, and robust alternatives such as the Huber loss can also be employed. All of these estimators will possess the non crossing property, if the same tuning parameter is used for different quantile levels.  
\end{remark}

\section{Local Rate of Convergence for Quantile TVD}\label{sec:ptwise_error}
In this section, our goal is to derive new local rates of convergence for the quantile TVD estimator under the fixed-design univariate quantile sequence model by utilizing the pointwise characterization of quantile TVD. To this end, let us first recall the definition of the quantile sequence model: the data points $y_1,y_2,\cdots,y_n$ are independent and the conditional $\tau$-th quantile ($\tau\in (0,1)$) of the $i$-th data point given the value of the corresponding design point $x_i$ is given by $f^*(x_i)$ for some unknown function $f^*:\bR \mapsto \bR$. It is equivalent to saying 
\[y_i=f^*(x_i)+\epsilon_i,\;\;1\leq i \leq n,\]
    where $x_1,\cdots, x_n$ are design points in $[0,1]$ and the errors $\epsilon_i$ s are independent such that $\tau$-th quantile of $\epsilon_i$ is $0$ and $\tau \in (0,1)$. The goal is to estimate the true quantile curve $f^*$ under mild structural assumptions. In particular, when the design points are fixed, equispaced points in $[0,1]$ (i.e., $x_i=i/n$), it naturally makes sense to consider the quantile sequence model:
    \[y_i=\theta^*_i+\epsilon_i,\;\;1\leq i \leq n,\]
    where $\theta^*=f^*(i/n)$ and the goal here is to estimate the vector $\theta^*=(\theta^*_1,\theta^*_2,\cdots,\theta^*_n)\in \mathbb{R}^n$. 
    For some recent works related to quantile sequence model, see \cite{madrid2022risk,madrid2024quantile}, etc.\\
    \\
    We first briefly exhibit how the pointwise characterization leads to bounds on the pointwise errors of quantile TVD under the quantile sequence model. To this end, fix some $i\in [n]$. Observe that the following deterministic inequalities hold:
    \[k-\text{th order statistic of }\{(a_i+b_i)\}_{i=1}^{K}\leq \max_{i \in [K]} a_i+  k-\text{th order statistic of }\{b_i\}_{i=1}^{K},\]
     \[k-\text{th order statistic of }\{(a_i+b_i)\}_{i=1}^{K}\geq \min_{i \in [K]} a_i+  k-\text{th order statistic of }\{b_i\}_{i=1}^{K},\]
    where we follow the notion of order statistic defined in Definition~\ref{defn:Y_I,k}. 
    Therefore, since $y_i=\theta^*_i+\epsilon_i$, we have
     \begin{align}\label{eqn:bias_sd}
   & \hat{\theta}_i-\theta^*_i\nonumber\\
 &\leq  \min_{J \in \I_i}\max_{I\subseteq J: i \in I}\{y_{I,(\lfloor u_{I,J}\rfloor+1)}-\theta^*_i\}\nonumber\allowdisplaybreaks\\
 &\leq  \min_{J \in \I_i}\max_{I\subseteq J: i \in I}\{\max_{k \in I}(\theta^*_k-\theta^*_i)+\epsilon_{I,(\lfloor u_{I,J}\rfloor+1)}.\}\nonumber\allowdisplaybreaks\\
 &\leq  \min_{J \in \I_i}\left[\underbrace{\max_{k \in J} (\theta^*_k-\theta^*_i)}_{\text{ Positive bias }}+\underbrace{\max_{I\subseteq J: i \in I}\,\epsilon_{I,(\lfloor u_{I,J}\rfloor+1)}}_{\lambda-\text{dependent stochastic term }}\right].
\end{align}
In a similar fashion, we can conclude that
\begin{align}\label{eqn:bias_sd_2}
  \hat{\theta}_i-\theta^*_i\geq  \max_{J \in \I_i}\left[\underbrace{\min_{k \in J} (\theta^*_k-\theta^*_i)}_{\text{ Negative bias }}+\underbrace{\min_{I\subseteq J: i \in I}\,\epsilon_{I,(\lceil l_{I,J}\rceil)}}_{\lambda-\text{dependent stochastic term }}\right] 
\end{align}

\begin{remark}
    Note that the above decompositions can be interpreted as quantile TVD doing a local (non standard) multiscale bias-variance tradeoff. At this point, it should be mentioned that a similar decomposition was also seen in the mean TVD setting (see \cite{chatterjee2026minmaxtrendfilteringgeneralizations}), but with certain structural differences both in the bias and variance terms. In particular, the stochastic term there took the form of maximum or minimum of averages of the noise variables over different subintervals $I$. This structural difference makes our subsequent probabilistic analysis substantially different from \cite{chatterjee2026minmaxtrendfilteringgeneralizations}.
\end{remark}

In order to provide high probability bounds on the stochastic terms, we need to make some assumption on the distributions of the errors. We make the following assumption on the growth rate of the underlying CDF around the true quantile.
\begin{assumption}\label{ass:growth_condn}
   For some $0<\delta<\infty$ independent of $n$, there exists  constants $c_1>0$ such that for each $1\leq k \leq n$,
    \[c_1|t|\leq |F_k(t)-\tau|,\;\;\text{ for } |t|\leq \delta,\]
    where $F_k$ is the CDF of $\epsilon_k$.
\end{assumption}
    \begin{remark}
        Assumption~\ref{ass:growth_condn} ensures that the $\tau$-th quantile of the error $\epsilon_k$ is uniquely defined. It is a mild assumption in the analysis of quantile regression estimators, see Assumption A of \cite{madrid2022risk} and the subsequent discussion therein. A simple case where this assumption is satisfied is
that the $\epsilon_i$ are independent draws from any density with respect to the Lebesgue measure that is bounded away
from zero on any compact interval (includes heavy-tailed distributions like Cauchy). In contrast with the mean version of TVD considered in \cite{chatterjee2026minmaxtrendfilteringgeneralizations}, where subgaussianity of the errors was assumed, Assumption~\ref{ass:growth_condn} does not put any restriction on the tail decay of the CDF.
    \end{remark}
Now, we introduce the following notions of local bias and standard deviation. 
\begin{definition}\label{defn:local_bias_sd}
    Fix a signal $\theta^* \in \bR^n$. Let $i\in [1:n]$ be any location and $J$ be any discrete subinterval of $[1:n]$ such that $i \in J$. Let us define
    \[Bias_+(i,J,\theta^*):=\max_{k \in J}(\theta^*_k-\theta^*_i);\]
    \[Bias_-(i,J,\theta^*):=\min_{k \in J}(\theta^*_k-\theta^*_i);\]
    \[SD^{\tau}(i,J,\lambda):=\tilde{C}\left(\sqrt{\frac{\log n}{Dist(i,\partial J)}}+\tau \frac{\log n}{\lambda}+\frac{\lambda}{|J|}\right),\]
    where for an interval $J=\left[j_1: j_2\right] \subseteq[n]$, we denote its boundary points by $\partial J=\left\{j_1, j_2\right\}$ and define
$$
\operatorname{Dist}(i, \partial J)=\begin{cases}
    \min \left\{i-j_1+1, j_2-i+1\right\} & \text{ if } C_1 \log n\leq i \leq n-C_1 \log n;\\
     j_2-i+1 & \text{ if } i<C_1 \log n;\\
    i-j_1+1 & \text{ if }i >n-C_1 \log n,
\end{cases} 
$$
and the constants $\tilde{C}$ and $C_1$ are some universal constants.
\end{definition}
While such notions of bias or standard deviation are not common, they arise naturally out of the decomposition in \eqref{eqn:bias_sd}, \eqref{eqn:bias_sd_2} and our proof. We are now ready to state our pointwise risk bound, which holds \textit{simultaneously} for all locations $i\in [n]$ with polynomially high probability.
\begin{proposition}\label{prop:ptwise_error_bound_interior}
 Under the quantile sequence model with some $\tau \in (0,1)$, suppose Assumption \ref{ass:growth_condn} holds. Given any $c>1$, let $$C:=\frac{(c+3)\tau}{8c_1} \text{ and }C_1:=\frac{c+2}{2c_1^2\delta^2}.$$
   Then, there exists a large enough universal constant $\tilde{C}>0$ and a natural number $N_0$ (depends on $c,\, c_1,\,\tau$) such that for any $n\geq N_0$ and $\lambda\geq C\log n$, the following bounds hold for the estimator defined in \eqref{eqn:objective_fn} at \textit{all locations} $i=\lceil C_1 \log n\rceil,\ldots,  \lfloor n-C_1 \log n\rfloor$ \textit{simultaneously} with probability $\geq 1-4n^{-(c-1)}:$
    \[\tilde{L}_i\leq \hat{\theta}_i-\theta_i\leq \tilde{U}_i,\]
    where
    \[\tilde{L}_i:=\max_{\substack{J \in \mathcal{I}_i: J \subseteq [2:n-1] \\ |J| > 4\lambda/(c_1\delta),\; \text{Dist}(i, \partial J) \ge C_1 \log n}} Bias_-(i,J,\theta^*)-SD^{1-\tau}(i,J,\lambda),\]
    \[\tilde{U}_i:=\min_{\substack{J \in \mathcal{I}_i: J \subseteq [2:n-1] \\ |J| > 4\lambda/(c_1\delta),\; \text{Dist}(i, \partial J) \ge C_1 \log n}} Bias_+(i,J,\theta^*)+SD^{\tau}(i,J,\lambda),\]
and $\mathcal{I}_i$ is the set of discrete sub-intervals of $[1:n]$ containing $i$.
\end{proposition}
\begin{remark}
    The main task in the proof is to derive high probability deterministic bound on the stochastic terms in \eqref{eqn:bias_sd} and \eqref{eqn:bias_sd_2}. A similar task was carried out in \cite{chatterjee2026minmaxtrendfilteringgeneralizations}, but they applied maximal inequality under subgaussianity, which is not applicable in our setup. We bound the stochastic term by first connecting the order statistics of the errors to appropriate indicator random variables, then applying Hoeffding inequality on these indicator random variables and doing a careful casework on the constants $C_{I,J}$.
\end{remark}
\begin{remark}
    Note that while we write down the above risk bound for $i=\lceil C_1 \log n\rceil,\ldots,  \lfloor n-C_1 \log n\rfloor$ only, similar risk bounds hold for other possible locations, with the difference being that the minimum (or maximum) is taken over a different collection of intervals. We state and prove a general version of Proposition~\ref{prop:ptwise_error_bound_interior} that includes all the above cases in Appendix~\ref{appnc}. However, for the analysis of the estimator at an interior design point $0<x_0<1$, Proposition~\ref{prop:ptwise_error_bound_interior} is sufficient, as it will be shown afterwards.
\end{remark}
    \noindent Now, fix any $0\leq x_0\leq 1$ and let $\hat{f}(x_0):=\hat{\theta}_{\lfloor nx_0\rfloor}$ be the corresponding quantile TVD estimate. In order to derive convergence rate of $\hat{f}(x_0)$, one needs to control the bias term inside the pointwise risk bounds in Proposition~\ref{prop:ptwise_error_bound_interior}. Thus, we need to make some structural assumption on the true signal $f^*$. In particular, we would assume $f^*$ to be locally H\"{o}lder continuous.
\begin{definition}[H\"{o}lder Continuous Functions]
    Given any subinterval $\mathbf{I} \subseteq[0,1]$, $\alpha \in[0,1]$,  we define the space of H\"{o}lder Continuous Functions, $C^{1, \alpha}(\mathbf{I})$, as the class of functions $f:[0,1] \rightarrow \mathbb{R}$ that are continuous on $\mathbf{I}$ and 
$$
|f|_{\mathbf{I} ; 1, \alpha} \stackrel{\text { def. }}{=} \sup _{x \neq y \in \mathbf{I}} \frac{\left|f(x)-f(y)\right|}{|x-y|^\alpha}<\infty .
$$
We call $|f|_{\mathbf{I} ; 1, \alpha}$ the H\"{o}lder coefficient (or norm) of $f$ on $\mathbf{I}$. If the above holds for some $\alpha>1$, then $f$ is constant on $I$. For notational continuity, we denote this case by $C^{1, \infty}(\mathbf{I})$ and set $|f|_{\mathbf{I} ; 1, \infty}=0$.
\end{definition}  
\begin{assumption}[Local H\"{o}lder continuity of signal]\label{ass:holder_cts}
There exists $r_0>0$, $\alpha_0>0$ such that
\begin{list}{}{%
    \settowidth{\labelwidth}{\textbf{Boundary Points:}}
    \setlength{\leftmargin}{\labelwidth}
    \addtolength{\leftmargin}{\labelsep}
}
    \item[\textit{Interior Point:}] if $0<x_0<1$, then $f^*$ is $\alpha_0$-H\"{o}lder continuous in the interval $[x_0\pm r_0]\subseteq [0,1]$ with H\"{o}lder norm $L_0$;
    \item[\textit{Boundary Points:}] if $x_0=0$ (or $=1$), then $f^*$ is $\alpha_0$-H\"{o}lder continuous in the interval $[0, r_0]\text{ (or }[r_0,1],$\\
    $\text{ respectively}) \subseteq [0,1]$ with H\"{o}lder norm $L_0$.
\end{list}
\end{assumption}
    We are now ready to state the main result of this section.

\begin{theorem}[Local Rate of Convergence]\label{thm:local_rate_holder}
 Fix any $0\leq x_0\leq 1$ and let $i_0:=\lfloor n x_0\rfloor$. Under the quantile sequence model with some $0<\tau <1$, suppose Assumption \ref{ass:growth_condn} and \ref{ass:holder_cts} hold.
Then, for any given $c>1$, there exist universal constants $C,C',C'', C'''>0$, such that for large enough $n\geq 1$,
\begin{itemize}
    \item [(A)] when $\alpha_0\leq 1$ and $ C \log n\leq \lambda <C'' n^{\frac{2\alpha_0}{2\alpha_0+1}}\,(\log n)^{\frac{1}{2\alpha_0+1}},$ the following holds with probability not less than $1-4n^{-c} $:
    \[|\hat{\theta}_{i_0}-\theta^*_{i_0}|\leq C'\left(\frac{\log n}{\lambda}+\frac{\lambda}{B_n}\right),\]
    where 
   \[B_n=
     \lfloor L_0^{-\frac{2}{2\alpha_0+1}}n^{\frac{2\alpha_0}{2\alpha_0+1}}(\log n)^{\frac{1}{2\alpha_0+1}}\rfloor.\]
   Moreover, the optimal choice of $\lambda$ is 
\[\lambda^*:=(\log n\;B_n)^{1/2},\]
and the corresponding local rate of convergence is \(O\left((2\alpha_0 L_0)^{\frac{1}{2\alpha_0+1}}n^{\frac{-\alpha_0}{2\alpha_0+1}}(\log n)^{\frac{\alpha_0}{2\alpha_0+1}}\right).\)
\item [(B)] when $\alpha_0>1$, i.e., the function $f^*$ is locally constant, and $C \log n\leq \lambda< C''' n,$ then
 \[|\hat{\theta}_{i_0}-\theta^*_{i_0}|\leq C'\left(\frac{\log n}{\lambda}+\frac{\lambda}{nr_0}\right).\]
 Moreover, the optimal choice of $\lambda$ is 
\[\lambda^*:=(nr_0\log n)^{1/2},\]
and the corresponding local rate of convergence is \(O\left((\log n/(nr_0))^{1/2}\right).\)
\end{itemize}
\end{theorem}

\begin{remark}
Theorem~\ref{thm:local_rate_holder} establishes consistency of quantile TVD for locally H\"{o}lder continuous signals for any positive smoothness level $\alpha_0$ and it explicitly shows how the rate of convergence depends on the local smoothness level $\alpha_0$. \cite{zhang2023element} is the only existing work we are aware of regarding pointwise error bound of quantile TVD, whose results are valid only for piecewise constant signals (Case $(B)$ in Theorem~\ref{thm:local_rate_holder}). Thus, our result appears to be the first of its kind that that hold for general $\alpha_0$. It should be noted that in the case of piecewise constant signals, we recover the local rate of convergence and the corresponding optimal choice of $\lambda$ that were derived in \cite{zhang2023element} at any interior point $x_0\in (0,1)$. Additionally, our results show consistency of the estimator at boundary points $x_0=0,1$, which was not possible in the analysis of \cite{zhang2023element}. 
\end{remark}

\begin{proof}[of Theorem~\ref{thm:local_rate_holder}:]
    We only prove the required risk bound for the case when $0<x_0<1$ as the other cases can be proved similarly using Proposition~\ref{prop:ptwise_error_bound}.  Let $i_0=\lfloor n r_0 \rfloor$ and set $C''=\frac{c_1\delta }{4}L_0^{-\frac{2}{2\alpha_0+1}}$. Because of the assumption $r_0>0$, for large enough $n$, it holds that $i_0\in [\lceil C_1\log n\rceil:\, \lfloor n-C_1\log n\rfloor]$ and also there exists at least one $J$ containing $i_0$ such that $Dist(i,\partial J)\geq C_1 \log n$, and $J$ is contained inside $\left[\lceil n(x_0-r_0)\rceil:\,\lfloor n(x_0+r_0)\rfloor\right].$ 
    Hereon, we divide the proof into two cases.\\
    \\
    \noindent \textit{Case (A):} Because of H\"{o}lder continuity, for such a $J$,
    \[Bias_+(i,J,\theta^*)\leq L_0 \frac{|J|^{\alpha_0}}{n^{\alpha_0}}.\]
    Thus, the term inside the minimum in Proposition~\ref{prop:ptwise_error_bound_interior} can be bounded by
    \[L_0 \frac{|J|^{\alpha_0}}{n^{\alpha_0}}+\left(\frac{\log n}{Dist(i,\partial J)}\right)^{1/2}+\frac{\log n}{\lambda}+\frac{\lambda}{|J|},\]
    upto a constant factor. Now, the sum of the first two terms is minimized when we choose $J$ to be an interval symmetric around $i_0$ such that 
    \[|J|=B_n=\min \left\{
     \lfloor L_0^{-\frac{2}{2\alpha_0+1}}n^{\frac{2\alpha_0}{2\alpha_0+1}}(\log n)^{\frac{1}{2\alpha_0+1}}\rfloor , 2nr_0\right\}.\]
     Note that for large enough $n$, the second quantity inside the above minimum dominates and hence
     \[|J|=B_n=\lfloor L_0^{-\frac{2}{2\alpha_0+1}}n^{\frac{2\alpha_0}{2\alpha_0+1}}(\log n)^{\frac{1}{2\alpha_0+1}}\rfloor.\]
    At this point, we would like to point out that such choice of $J$ satisfies the constraint $|J|>4\lambda/(c_1 \delta)$ as 
    \[\frac{4\lambda}{c_1 \delta}<\frac{4C''}{c_1\delta} n^{\frac{2\alpha_0}{2\alpha_0+1}}(\log n)^{\frac{1}{2\alpha_0+1}}=L_0^{-\frac{2}{2\alpha_0+1}}n^{\frac{2\alpha_0}{2\alpha_0+1}}(\log n)^{\frac{1}{2\alpha_0+1}}.\]
    For this choice of $J$, the sum of the first two terms in the pointwise estimation error bound can be bounded by
    \[L_0^{\frac{1}{2\alpha_0+1}} n^{-\frac{\alpha_0}{2\alpha_0+1}} (\log n)^{\frac{\alpha_0}{2\alpha_0+1}}.\]
    Let us define
    \[g(\lambda):=\frac{\log n}{\lambda}+\frac{\lambda}{B_n}.\]
    $g$ is minimized at $\lambda^*=(\log n B_n)^{1/2}$ (for large enough $n$, $C \log n\leq \lambda^*\leq C'' n^{\frac{2\alpha_0}{2\alpha_0+1}}(\log n)^{\frac{1}{2\alpha_0+1}}$) and $g(\lambda^*)=(2\log n/B_n)^{1/2}$ is at least of the order of \(L_0^{\frac{1}{2\alpha_0+1}} n^{-\frac{\alpha_0}{2\alpha_0+1}} (\log n)^{\frac{\alpha_0}{2\alpha_0+1}}.\)\\
    \\
   \noindent \textit{Case (B):} Because $f^*$ is constant in such a $J$, $Bias_+(i,J,\theta^*)=0$.  Thus, the quantity inside the minimum in Proposition~\ref{prop:ptwise_error_bound} can be bounded by
\begin{equation}\label{eqn:constant_case_upper_bound}
    \left(\frac{\log n}{Dist(i,\partial J)}\right)^{1/2}+\frac{\log n}{\lambda}+\frac{\lambda}{|J|},
\end{equation}
    upto a constant factor.
    For any fixed $\lambda$, the above quantity is minimized when $J$ is taken to be the interval $[\lceil n(x_0-r_0)\rceil:\,\lfloor n(x_0+r_0)\rfloor]$ and for this choice of $J$,~\eqref{eqn:constant_case_upper_bound} can be bounded by
    \[\left(\frac{\log n}{nr_0}\right)^{1/2}+\frac{\log n}{\lambda}+\frac{\lambda}{nr_0},\]
    upto a constant factor. Since $\lambda \geq C \log n$, the first term gets dominated by the third for large $n$. It is not difficult to see that this quantity is minimized by $\lambda^*=(nr_0 \log n)^{1/2}$ and the upper bound corresponding to $\lambda^*$ is \(O\left((\log n /(nr_0))^{1/2}\right).\)\\
    \\
    \noindent  Thus, the upper bound is established in both cases. The lower bound can be established similarly by deriving upper bound on $-\tilde{L}_i$, which is defined in Proposition~\ref{prop:ptwise_error_bound_interior}.
\end{proof}

\section{Summary}\label{sec:conclusion}

In this paper, we established an exact pointwise characterization of univariate quantile total variation denoising via minmax/maxmin representations of its fitted values. Despite the inherent non-uniqueness of the estimator, we showed that the set of admissible fitted values at each location forms a compact interval, whose endpoints admit explicit variational formulae in terms of local order statistics over nested intervals.

These representations yield several new insights. Structurally, they imply that quantile TVD is intrinsically non-crossing across quantile levels when a common tuning parameter is used, a property we further relate to the submodularity of the TV penalty and extend to a broader class of penalized quantile regression estimators. Statistically, the pointwise formulas lead to a local bias--variance decomposition, enabling the derivation of finite-sample pointwise risk bounds and near-optimal local rates under H\"older smoothness and heavy-tailed noise.

Overall, our results provide a unified and transparent framework for understanding the behavior of quantile TVD and highlight the broader utility of exact pointwise representations in the analysis of globally regularized estimators.

\bigskip
    
	\setlength{\bibsep}{0pt plus 0.25ex}
	
	\bibliographystyle{chicago}

    \bibliography{bibilography.bib}

    \begin{appendix}
    \refstepcounter{section}
\section*{Appendix \thesection: Lemmas required to prove Theorem~\ref{thm:minmax_bound}}\label{appna}
\begin{proof}[Proof of Lemma~\ref{lem:orderstat}]
(a) If $m<0$ or $m\geq |I|$, the bound holds trivially. So, let's assume $m \in \{0,1,\cdots,|I|-1\}.$
If $t>y_{I,(m+1)}$ then at least $m+1$ points in $I$ are $<t$, contradiction.
(b) If $m\leq 0$ or $m\geq |I|+1$, the bound holds trivially. So, let's assume $m \in \{1,\cdots,|I|\}.$
If $t<y_{I,(m)}$ then at most $m-1$ points are $\le t$, contradiction.
\end{proof}

\begin{proof}[Proof of Lemma~\ref{lemma:z_upper_bd}]
   We verify \eqref{eq:z-upper-key-thm1} case by case, exactly following Definition~\ref{defn:C_{IJ}}.
Recall that $|z_k|\le\lambda$ for all $k=1,\dots,n-1$, and $z_0=z_n=0$. Moreover, when
$\hat\theta_{k+1}\neq\hat\theta_k$, we have $z_k=\pm\lambda$ with sign determined by
$\operatorname{sign}(\hat\theta_k-\hat\theta_{k+1})$.

\medskip
\noindent\underline{\textbf{Case 1: $1<a\le b<n$ (interior $J$).}}
\[
C_{I,J}=
\begin{cases}
1, & I\subset J,\\
-1, & I=J,\\
0, & \text{otherwise}.
\end{cases}
\]
\begin{itemize}
\item If $I\subset J$, maximality of $I$ forces $z_{c-1}=-\lambda$ and $z_d=\lambda$,
so $z_{c-1}-z_d=-2\lambda$.

\item If $I=[c:b]$ with $c > a$ then maximality of $I$ forces $z_{c-1}=-\lambda$, and thus
$z_{c-1}-z_d=-\lambda-z_d\le 0$ since $|z_d|\le \lambda.$

\item If $I=[a:d]$ with $d < b$, then maximality of $I$ forces $z_d=\lambda$, and thus
$z_{c-1}-z_d=z_{c-1}-\lambda\le 0$ since $|z_{c-1}|\le \lambda.$

\item If $I=J$, $z_{c-1}-z_d = z_{a-1}-z_b \leq 2\lambda$ since $\max\{|z_{a - 1}|,|z_{b}|\} \leq \lambda.$
\end{itemize}

\medskip
\noindent\underline{\textbf{Case 2: $J=[1:b]$, $b<n$.}}
\[
C_{I,J}=
\begin{cases}
1, & I\subset J,\\
-\tfrac12, & I=J,\\
\tfrac12, & I=[1:d],\ d<b,\\
0, & I=[c:b],\ c>1.
\end{cases}
\]
Using $z_0=0$:
\begin{itemize}
\item $I\subset J$: same as Case 1, $z_{c-1}-z_d=-2\lambda$.
\item $I=[1:d]$, $d<b$: maximality of $I$ forces $z_d=\lambda$, hence $z_{c-1}-z_d=-z_d=-\lambda$.
\item $I=[c:b]$, $c>1$: maximality of $I$ forces $z_{c-1}=-\lambda$, hence $z_{c-1}-z_b=-\lambda-z_b \le 0$ since $|z_b|\le \lambda$.
\item $I=J$: $z_{c-1}-z_d=-z_b\le\lambda$.
\end{itemize}

\medskip
\noindent\underline{\textbf{Case 3: $J=[a:n]$, $a>1$.}}
\[
C_{I,J}=
\begin{cases}
1, & I\subset J,\\
-\tfrac12, & I=J,\\
\tfrac12, & I=[c:n],\ c>a,\\
0, & I=[a:d],\ d<n.
\end{cases}
\]
Using $z_n=0$:
\begin{itemize}
\item $I\subset J$: same as Case 1, $z_{c-1}-z_d=-2\lambda$.
\item $I=[c:n]$ with $c>a$: maximality of $I$ forces $z_{c-1}=-\lambda$, hence $z_{c-1}-z_d=z_{c-1}-z_n=z_{c-1}=-\lambda$.
\item $I=[a:d]$ with $d<n$: maximality of $I$ forces $z_d=\lambda$, hence $z_{c-1}-z_d=z_{a-1}-\lambda\le 0$ since $|z_{a-1}|\le\lambda$.
\item $I=J$: $z_{c-1}-z_d=z_{a-1}-z_n=z_{a-1}\le \lambda$.
\end{itemize}

\medskip
\noindent\underline{\textbf{Case 4: $J=[1:n]$.}}
\[
C_{I,J}=
\begin{cases}
1, & I\subset J,\\
0, & I=J,\\
\tfrac12, & \text{otherwise}.
\end{cases}
\]
Using $z_0=z_n=0$:
\begin{itemize}
\item $I\subset J$: same as Case 1, $z_{c-1}-z_d=-2\lambda$.
\item If $I=[c:d]$ with $c > 1, d = n$ then maximality of $I$ forces $z_{c-1}=-\lambda$, thus giving $z_{c-1}-z_d = z_{c-1} = -\lambda.$
\item If $I=[c:d]$ with $c = 1,d < n$, then maximality of $I$ forces $z_d=\lambda$, thus giving $z_{c-1}- z_d = -z_d = -\lambda.$
\item If $I=J$, then $z_{c-1}-z_d  = z_0 - z_n = 0.$
\end{itemize}

This exhausts all the possible cases and thus proves \eqref{eq:z-upper-key-thm1}. 
\end{proof}

\begin{proof} [Proof of Lemma~\ref{lemma:z_lower_bd}]
    We verify \eqref{eq:z-lower-key-thm2} case by case, exactly following Definition~\ref{defn:C_{IJ}}.
In addition to $|z_k|\le\lambda$, we use the maximality of $\hat{J}$ which implies:
if $a>1$ then $\hat\theta^{\max}_{a-1}>\hat\theta^{\max}_a$, hence $z_{a-1}=+\lambda$;
if $b<n$ then $\hat\theta^{\max}_{b}<\hat\theta^{\max}_{b+1}$, hence $z_b=-\lambda$.

\medskip
\noindent\underline{\textbf{Case 1: $1<a\le b<n$ (interior $\hat{J}$).}}
\[
C_{I,J}=
\begin{cases}
1, & I\subset \hat{J},\\
-1, & I=\hat{J},\\
0, & \text{otherwise}.
\end{cases}
\]
\begin{itemize}
\item If $I\subset \hat{J}$, then $z_{c-1}-z_d \geq -2\lambda$ since $\max\{|z_{c - 1}|,|z_{d}|\} \leq \lambda.$

\item If $I=[c:d]$ with $a = c, d < b$, then $z_{c - 1} = z_{a-1}=\lambda$, giving $z_{c-1}- z_d = \lambda - z_d \ge 0$ since $|z_{d}| \leq \lambda.$

\item If $I=[c:d]$ with $c > a, d = b$ then $z_{d} = z_b=-\lambda$, giving $z_{c-1}-z_d = z_{c-1}+\lambda \ge 0$ since $|z_{c - 1}| \leq \lambda.$

\item If $I=\hat{J}$, then $z_{c-1}-z_d = z_{a-1}-z_b = \lambda-(-\lambda)=2\lambda.$
\end{itemize}

\medskip
\noindent\underline{\textbf{Case 2: $\hat{J}=[1:b]$, $b<n$.}}
\[
C_{I,\hat{J}}=
\begin{cases}
1, & I\subset \hat{J},\\
-\tfrac12, & I=\hat{J},\\
\tfrac12, & I=[1:d],\ d<b,\\
0, & I=[c:b],\ c>1.
\end{cases}
\]
Using $z_0=0$ and $z_b=-\lambda$:
\begin{itemize}
\item $I\subset \hat{J}$: same as Case 1, $z_{c-1}-z_d \geq -2\lambda$.
\item $I=[c:d]$, $c = 1, d<b$: $z_{c-1}-z_d=-z_d \geq -\lambda$.
\item $I=[c:d]$, $c>1, d = b$: $z_{c-1}-z_d = z_{c-1}+\lambda \ge 0$.
\item $I=\hat{J}$: $z_{c-1}-z_d=-z_b = \lambda$.
\end{itemize}

\medskip
\noindent\underline{\textbf{Case 3: $\hat{J}=[a:n]$, $a>1$.}}
\[
C_{I,\hat{J}}=
\begin{cases}
1, & I\subset \hat{J},\\
-\tfrac12, & I=\hat{J},\\
\tfrac12, & I=[c:n],\ c>a,\\
0, & I=[a:d],\ d<n.
\end{cases}
\]
Using $z_n=0$ and $z_{a-1}=+\lambda$:
\begin{itemize}
\item $I\subset \hat{J}$: same as Case 1, $z_{c-1}-z_d \geq -2\lambda$.
\item $I=[c:d]$ with $c>a, d=n$: $z_{c-1}-z_d=z_{c-1}\ge -\lambda$.
\item $I=[c:d]$ with $c=a, d<n$: $z_{c-1}-z_d=z_{a-1}-z_d\ge 0$.
\item $I=\hat{J}$: $z_{c-1}-z_d=z_{a-1}-z_n=z_{a-1}=\lambda$.
\end{itemize}

\medskip
\noindent\underline{\textbf{Case 4: $\hat{J}=[1:n]$.}}
\[
C_{I,\hat{J}}=
\begin{cases}
1, & I\subset \hat{J},\\
0, & I=\hat{J},\\
\tfrac12, & \text{otherwise}.
\end{cases}
\]
Using $z_0=z_n=0$:
\begin{itemize}
\item $I\subset \hat{J}$: same as Case 1, $z_{c-1}-z_d \geq -2\lambda$.
\item If $I=[c:d]$ with $c > 1, d = n$ then $z_d=0$ so $z_{c-1}-z_d = z_{c-1} \geq -\lambda.$
\item If $I=[c:d]$ with $c = 1,d < n$, then $z_{c-1}=0$ so $z_{c-1}- z_d = -z_d \geq -\lambda.$
\item If $I=\hat{J}$, then $z_{c-1}-z_d  = z_0 - z_n = 0.$
\end{itemize}

This exhausts all cases and proves \eqref{eq:z-lower-key-thm2}.
\end{proof}
\begin{lemma}\label{lem:minmax_finite}
    Fix any location $i\in [n]$. Let 
    \[L_i=\max_{J\in \I: i \in J}\min_{I\in \I: I\subseteq J,\, i \in I}\, y_{I,(\lceil l_{I,J}\rceil)},\quad U_i=\min_{J\in \I: i \in J}\max_{I\in \I: I\subseteq J,\, i \in I}\, y_{I,(\lfloor u_{I,J}\rfloor+1)}.\]
    Then, the following hold:
    \begin{enumerate}
        \item If $\tau \in (0,1)$, then
        \begin{equation*}\label{eqn:minmax_finite}
-\infty<L_i\leq U_i  <\infty.
 \end{equation*}
 \item If $\tau=0$ (or $1$) and $\lambda>0$, then
  \begin{align*}\label{eqn:minmax_infinite}
 L_i=-\infty\text{ ( or }\max_{i \in [n]}y_i,\text{ respectively});\quad U_i =\min_{i \in [n]}y_i\text{ (or }\infty,\text{ respectively)}.
 \end{align*}
    \end{enumerate}
    
    \end{lemma}
\begin{proof}  
  $(1).$ We first prove $L_i\leq U_i$ from first principle. It suffices to show that for any $J_1,J_2 \in \I$ containing $i$, it holds that
  \begin{equation}\label{eqn:min_leq_max}
    \min_{I\in \I: I\subseteq J_1,\, i \in I}\, y_{I,(\lceil l_{I,J_1}\rceil)}\leq \max_{I\in \I: I\subseteq J_2,\, i \in I}\, y_{I,(\lfloor u_{I,J_2}\rfloor+1)}.  
  \end{equation}
  Note that~\eqref{eqn:min_leq_max} compares two lists of numbers that are not necessarily identical. In order to establish~\eqref{eqn:min_leq_max}, it thus suffices to show that there is a common element in both the lists, i.e., there exists $I\subseteq J_1,J_2$ that contains $i$ and
  \[y_{I,(\lceil l_{I,J_1}\rceil)}\leq y_{I,(\lfloor u_{I,J_2}\rfloor+1)}.\]
  To this end, consider $I=J_1\cap J_2.$ Since both $J_1,J_2$ contain $i$, so does $I$. Now, if $J_1\subset J_2$ (in a strict sense), then by Definition~\ref{defn:C_{IJ}}, $C_{I,J_1}\leq 0$ and $C_{I,J_2}\geq 0$. Thus,
  \[l_{I,J_1}=\tau|I|+2\lambda \,C_{I,J_1}\leq \tau |I|,\quad u_{I,J_2}=\tau|I|-2\lambda \,C_{I,J_2}\geq \tau|I|\Rightarrow \lceil l_{I,J_1} \rceil \leq \lfloor u_{I,J_2}\rfloor+1,\]
  which in turn implies 
  \[y_{I,(\lceil l_{I,J_1}\rceil)}\leq y_{I,(\lfloor u_{I,J_2}\rfloor+1)}.\]
  The case of $J_2\subset J_1$ is exactly similar by symmetry. Now, when $J_1=J_2$, $C_{I,J_1}=C_{I,J_2}\leq 0.$ Thus, by our previous analysis, we again get
  \[y_{I,(\lceil l_{I,J_1}\rceil)}\leq y_{I,(\lfloor u_{I,J_2}\rfloor+1)},\]
  and thus, we have proved $L_i\leq U_i$.\\
  \\
  \noindent 
  We next show that $U_i<\infty$.  It suffices to show that there exists at least one $J\subseteq [n]$ containing $i$ such that  \[\max_{I\in \I: I\subseteq J,\, i \in I}\, y_{I,(\lfloor u_{I,J}\rfloor+1)}<\infty.\]
  Now, note that for any $I\subseteq [1:n]$, $u_{I,[1:n]}<|I|$ because of the definition of $C_{I,[1:n]}$ and $\tau<1$. Therefore, 
 \[\max_{I\in \I: I\subseteq [1:n],\, i \in I}\, y_{I,(\lfloor u_{I,[1:n]}\rfloor+1)}<\infty.\]
 Thus, we have proved $U_i<\infty.$
\\
\noindent We next prove that $L_i>-\infty$. It suffices to show that there exists at least one $J\subseteq [n]$ containing $i$, such that
 \[\min_{I\in \I: I\subseteq J,\, i \in I}\, y_{I,(\lceil l_{I,J}\rceil)}>-\infty.\] 
 Again, note that for any $I\subseteq [1:n]$, $l_{I,[1:n]}>0$ by the definition of $C_{I,[1:n]}$ and the fact $\tau>0$. Therefore, $\lceil l_{I,[1:n]}\rceil \geq 1$ and hence,
 \[ y_{I,(\lceil l_{I,[1:n]}\rceil)}>-\infty.\]
 Thus, we have proved $L_i>-\infty.$\\
\\
\noindent $(2).$ Let's compute the envelope bounds in the case of $\tau=0$. Note that in this case, by Definition~\ref{defn:C_{IJ}}, for any $J \in \cI$ containing $i$, there always exists a $I\subseteq J$ containing $i$ such that $C_{I,J}=0$. Therefore, for that particular $I$, $l_{I,J}=0$ and hence, by Definition~\ref{defn:Y_I,k}, we have \(\min_{I\in \I: I\subseteq J,\, i \in I}\,y_{I,(\lceil l^{\tau}_{I,J}\rceil)}=-\infty.\) Since this holds for any such choice of $J$, we conclude that
\[\max_{J\in \I: \, i \in I}\,\min_{I\in \I: I\subseteq J,\, i \in I}\,y_{I,(\lceil l^{\tau}_{I,J}\rceil)}=-\infty.\]
Next, note that when $J=[1:n]$, it follows from Definition~\ref{defn:C_{IJ}} that $u_{I,J}<0$ if $I\neq J$ and $u_{I,J}=0$ when $I=J$. Thus, $\lfloor u_{I,J}\rfloor+1\leq 0$ and $=1$ when $I\neq J$ and $I=J$, respectively. Thus, by Definition~\ref{defn:Y_I,k}, we have
\[\max_{I\in \I: I\subseteq J,\, i \in I}\,y_{I,(\lceil l^{\tau}_{I,J}\rceil)}=y_{[1:n],(1)}=\min_{i \in [n]}y_i.\]
Therefore, we conclude that $\min_{J\in \I: i \in J}\max_{I\in \I: I\subseteq J,\, i \in I}\, y_{I,(\lfloor u_{I,J}\rfloor+1)}\leq \min_{i \in [n]}y_i$. Moreover, by Definition~\ref{defn:C_{IJ}}, for any $J=[1:n]$, there exists at least one $I\subseteq J$ containing $i$ such that $C_{I,J}<0$ and hence, $u^{\tau}_{I,J}>0$ for $\tau=0$ and $\lambda>0$.  Thus, for such a $J$,
\[\max_{I\in \I: I\subseteq J,\, i \in I}\, y_{I,(\lfloor u_{I,J}\rfloor+1)}\geq y_{J,(1)}\geq \min_{i\in [n]}y_i.\]
Hence, we conclude that $\min_{J\in \I: i \in J}\max_{I\in \I: I\subseteq J,\, i \in I}\, y_{I,(\lfloor u_{I,J}\rfloor+1)}=\min_{i \in [n]}y_i.$ Thus, the statement of the theorem is proved for $\tau=0$. The case of $\tau=1$ can be handled in a similar fashion.\\
\end{proof}

\begin{proof}[Proof of Lemma~\ref{lemma:TV_is_affine}]

In order to prove~\eqref{eqn:penalty_term_is_affine}, we consider the four possible cases:
\begin{itemize}
    \item \underline{$I\subseteq [2:n-1]$:} In this case, 
    \begin{align*}
        &TV(\theta(t))-TV(\theta(0))\\
        &=|\theta_c(t)-\theta_{c-1}(t)|+|\theta_{d+1}(t)-\theta_d(t)|-|\hat{\theta}^{\max}_c-\hat{\theta}^{\max}_{c-1}|-|\hat{\theta}^{\max}_{d+1}-\hat{\theta}^{\max}_d|\\
        &=\left(|\hat\theta^{\max}_c+t-\hat\theta^{\max}_{c-1}|-|\hat{\theta}^{\max}_c-\hat{\theta}^{\max}_{c-1}|\right)+\left(|\hat\theta^{\max}_{d+1}-\hat\theta^{\max}_d-t|-|\hat{\theta}^{\max}_{d+1}-\hat{\theta}^{\max}_d|\right).\\
    \end{align*}
    Now, for small enough $t>0$, it holds that
    \begin{align*}
        |\hat\theta^{\max}_c-\hat\theta^{\max}_{c-1}+t|-|\hat{\theta}^{\max}_c-\hat{\theta}^{\max}_{c-1}|&=\begin{cases}
            \hat\theta^{\max}_c-\hat\theta^{\max}_{c-1}+t-\hat{\theta}^{\max}_c+\hat{\theta}^{\max}_{c-1} & \text{ if }  \hat\theta^{\max}_c\geq \hat\theta^{\max}_{c-1};\\
            \hat\theta^{\max}_{c-1}-\hat\theta^{\max}_c-t+\hat{\theta}^{\max}_c-\hat{\theta}^{\max}_{c-1} & \text{ if }  \hat\theta^{\max}_c< \hat\theta^{\max}_{c-1};\\
        \end{cases}\\
        &=\begin{cases}
            t & \text{ if }  \hat\theta^{\max}_c\geq \hat\theta^{\max}_{c-1};\\
            -t & \text{ if }  \hat\theta^{\max}_c< \hat\theta^{\max}_{c-1};\\
        \end{cases}\\\
        &=\sigma_{c-1}\,t,
    \end{align*}
    where in the first equality, we are using the fact that if $\hat\theta^{\max}_c-\hat\theta^{\max}_{c-1}\geq 0$ (or $<0$, respectively), then so is $\hat\theta^{\max}_c-\hat\theta^{\max}_{c-1}+t$ for small enough $t>0$.
    In a similar fashion, it is easy to conclude that 
    \[|\hat\theta^{\max}_{d+1}-\hat\theta^{\max}_d-t|-|\hat{\theta}^{\max}_{d+1}-\hat{\theta}^{\max}_d|=\sigma_d\,t.\]
   \item \underline{$I=[1:d]$, where $d<n$:} In this case, 
   \begin{align*}
        &TV(\theta(t))-TV(\theta(0))\\
        &=|\theta_{d+1}(t)-\theta_d(t)|-|\hat{\theta}^{\max}_{d+1}-\hat{\theta}^{\max}_d|\\
        &=|\hat\theta^{\max}_{d+1}-\hat\theta^{\max}_d-t|-|\hat{\theta}^{\max}_{d+1}-\hat{\theta}^{\max}_d|\\
        &=\sigma_d\,t,
    \end{align*}
    where the last equality follows holds for small enough $t>0$ from the same reasoning we used in the previous case.
    \item \underline{$I=[c:n]$, where $1<c$:} Similar to the previous case, one can show again that
    \[TV(\theta(t))-TV(\theta(0))=\sigma_{c-1}\,t,\]
    for small enough $t$.
    \item \underline{$I=[1:n]$:} In this case, \(TV(\theta(t))=TV(\theta(0))\) for any $t\geq 0$.
\end{itemize}
Thus, \eqref{eqn:penalty_term_is_affine} is proved.
\end{proof}
\begin{proof}[Proof of Lemma~\ref{lemma:tv-bound-full-lemma}]
We will use:
(i) whenever defined, $\sigma\in\{-1,+1\}$;
(ii) maximality of $\hat{J}=[a:b]$ implies
\[
z_{a-1}=\lambda\one \{a>1\},
\qquad
z_b=-\lambda\one \{b<n\},
\]
equivalently $\sigma_{a-1}=-1$ and $\sigma_b=+1$ whenever they are defined. 

\medskip
\noindent\underline{\textbf{Case 1: $1<a\le b<n$.}}
In this case,
\[\sigma_{c-1}\one\{c\neq 1\}-\sigma_d\one\{d\neq n\}=\sigma_{c-1}-\sigma_d,\]
and
\[
C_{I,\hat J}=
\begin{cases}
1,& I\subset \hat J,\\
-1,& I=\hat J,\\
0,& \text{otherwise}.
\end{cases}
\]

\begin{itemize}
\item If $I\subset \hat J$, then $\sigma_{c-1}-\sigma_d\le 2$, hence
$\lambda(\sigma_{c-1}-\sigma_d)\le 2\lambda=2\lambda C_{I,\hat J}$.

\item If $I=\hat J$, then $\sigma_{c-1}=\sigma_{a-1}=-1$ and $\sigma_d=\sigma_b=+1$,
so $\sigma_{c-1}-\sigma_d=-2=2C_{I,\hat J}$.

\item If $I=[c:b]$ with $c>a$, then $\sigma_d=+1$ and $\sigma_{c-1}\le 1$,
so $\sigma_{c-1}-\sigma_d\le 0=2C_{I,\hat J}$.

\item If $I=[a:d]$ with $d<b$, then $\sigma_{c-1}=-1$ and $\sigma_d\ge -1$,
so $\sigma_{c-1}-\sigma_d\le 0=2C_{I,\hat J}$.
\end{itemize}

\medskip
\noindent\underline{\textbf{Case 2: $\hat J=[1:b]$, $b<n$.}}
\[
C_{I,\hat J}=
\begin{cases}
1,& I\subset \hat J,\\
-\tfrac12,& I=\hat J,\\
\tfrac12,& I=[1:d],\,d<b,\\
0,& I=[c:b],\,c>1.
\end{cases}
\]

\begin{itemize}
\item $I\subset \hat J$: same as Case 1.

\item $I=\hat J$: $\sigma_d=+1$, left term absent,
so $\sigma_{c-1}\one\{c\neq 1\}-\sigma_d\one\{d\neq n\}==-1=2C_{I,\hat J}$.

\item $I=[1:d]$, $d<b$:
$\sigma_{c-1}\one\{c\neq 1\}-\sigma_d\one\{d\neq n\}=-\sigma_d\le 1=2C_{I,\hat J}$.

\item $I=[c:b]$, $c>1$:
$\sigma_{c-1}\one\{c\neq 1\}-\sigma_d\one\{d\neq n\}=\sigma_{c-1}-1\le 0=2C_{I,\hat J}$.
\end{itemize}

\medskip
\noindent\underline{\textbf{Case 3: $\hat J=[a:n]$, $a>1$.}}

Symmetric to Case 2.

\medskip
\noindent\underline{\textbf{Case 4: $\hat J=[1:n]$.}}

\[
C_{I,\hat J}=
\begin{cases}
1,& I\subset \hat J,\\
0,& I=\hat J,\\
\tfrac12,& \text{otherwise}.
\end{cases}
\]

\begin{itemize}
\item $I\subset \hat J$: $\sigma_{c-1}\one\{c\neq 1\}-\sigma_d\one\{d\neq n\}\leq 2=2C_{I,\hat J}$.
\item $I=\hat J$: both terms absent, equality.
\item $I=[1:d],\;d<n$: $\sigma_{c-1}\one\{c\neq 1\}-\sigma_d\one\{d\neq n\}=-\sigma_d\le 1=2C_{I,\hat J}$.
\item $I=[c:n],\;c>1$: $\sigma_{c-1}-\sigma_d=\sigma_{c-1}\le 1=2C_{I,\hat J}$.
\end{itemize}

Thus, \eqref{eq:tv-bound-full-lemma} holds. 
\end{proof}

\begin{proof}[Proof of Lemma~\ref{lem:J_hat_in_tilde_I}]
    Let, if possible, there exist an $I=[c:d]\subseteq \hat{J}$ such that $u_{I,\hat{J}}=\tau|I|-2\lambda\,C_{I,\hat{J}}\geq |I|$.
    However, by the definition of $C_{I,\hat{J}}$, $$u_{I,\hat{J}}\geq |I| \Rightarrow C_{I,\hat{J}}<0\Rightarrow I=\hat{J} \text{ and } \hat{J}\neq [1:n].$$
    Define
\[
\theta(t):=\hat\theta^{\max}+t\,\one_{\hat{J}},\qquad t\ge 0,
\]
and $\phi(t):=F(\theta(t))$, where $F(\theta)$ is defined in~\eqref{eqn:objective_fn}. For $j\in \hat{J}$, write
\[
u_j(t):=y_j-\theta_j(t)=y_j-\hat\theta^{\max}_j-t.
\]
If $y_j>\hat\theta^{\max}_j$, then $u_j(0)>0$ and the sign remains positive for
$t<y_j-\hat\theta^{\max}_j$.
If $y_j=\hat\theta^{\max}_j$, then $u_j(t)=-t<0$ for all $t>0$.
If $y_j<\hat\theta^{\max}_j$, then $u_j(t)<0$ for all $t\ge 0$.
Hence, letting
\[
\eta:=\min_{j\in \hat{J}:\,y_j>\hat\theta^{\max}_j}
(y_j-\hat\theta^{\max}_j),
\]
(with $\eta=+\infty$ if the index set is empty), we conclude that $u_j(t)$ does not change sign for $t\in [0,\eta]$ for any $j \in \hat{J}.$ Thus, from the definition of $\rho_{\tau}$, it follows that each $\rho_\tau(u_j(t))$ is affine on $[0,\eta]$,
and so is their sum. 
Thus,
\[
\frac{d}{dt}\rho_\tau(u_j(t))\Big|_{0^+}
=
\begin{cases}
-\tau, & \hat\theta^{\max}_j<y_j,\\
1-\tau, & \hat\theta^{\max}_j\ge y_j.
\end{cases}
\]
Summing gives
\begin{equation}\label{eq:data-deriv-full}
\frac{d}{dt}\sum_{j=a}^b \rho_\tau(y_j-\theta_j(t))\Big|_{0^+}
=
-\tau|\hat{J}|+\#\{j\in \hat{J}:\hat\theta^{\max}_j\ge y_j\}.
\end{equation}

Now, we turn our attention to the penalty term. Only the two boundary edges across $\partial \hat{J}$ depend on $t$ in the TV term. We consider the three possible cases:
\begin{itemize}
    \item \underline{$\hat{J}\subseteq [2:n-1]$:} In this case, $C_{\hat{J},\hat{J}}=-1$. Now, for small enough $t$, it holds that
    \begin{align*}
        &TV(\theta(t))-TV(\theta(0))\\
        &=|\theta_a(t)-\theta_{a-1}(t)|+|\theta_{b+1}(t)-\theta_b(t)|-|\hat{\theta}^{\max}_a-\hat{\theta}^{\max}_{a-1}|-|\hat{\theta}^{\max}_{b+1}-\hat{\theta}^{\max}_b|\\
        &=|\hat\theta^{\max}_a-\hat\theta^{\max}_{a-1}+t|+|\hat\theta^{\max}_{b+1}-\hat\theta^{\max}_b-t|-|\hat{\theta}^{\max}_a-\hat{\theta}^{\max}_{a-1}|-|\hat{\theta}^{\max}_{b+1}-\hat{\theta}^{\max}_b|\\
        &=\hat\theta^{\max}_{a-1}-\hat\theta^{\max}_a-t+\hat\theta^{\max}_{b+1}-\hat\theta^{\max}_b-t+\hat{\theta}^{\max}_a-\hat{\theta}^{\max}_{a-1}-\hat{\theta}^{\max}_{b+1}+\hat{\theta}^{\max}_b=-2t=2C_{\hat{J},\hat{J}}t,
    \end{align*}
    where the third equality follows from the fact that $\hat{\theta}^{\max}_{a-1}<\hat{\theta}^{\max}_{a}$ and $\hat{\theta}^{\max}_{b}<\hat{\theta}^{\max}_{b+1}$ and for small enough $t>0$, $\hat\theta^{\max}_a-\hat\theta^{\max}_{a-1}+t$ (or $\hat\theta^{\max}_{b+1}-\hat\theta^{\max}_b-t$) is negative (or positive). 
   \item \underline{$\hat{J}=[1:b]$, where $b<n$:} In this case, $C_{\hat{J},\hat{J}}=-1/2$. Now, small enough $t$,
   \begin{align*}
        &TV(\theta(t))-TV(\theta(0))\\
        &=|\theta_{b+1}(t)-\theta_b(t)|-|\hat{\theta}^{\max}_{b+1}-\hat{\theta}^{\max}_b|\\
        &=|\hat\theta^{\max}_{b+1}-\hat\theta^{\max}_b-t|-|\hat{\theta}^{\max}_{b+1}-\hat{\theta}^{\max}_b|\\
        &=\hat\theta^{\max}_{b+1}-\hat\theta^{\max}_b-t-\hat{\theta}^{\max}_{b+1}+\hat{\theta}^{\max}_b=-t=2C_{\hat{J},\hat{J}}t.
    \end{align*}
    \item \underline{$\hat{J}=[a:n]$, where $1<a$:} Similar to the previous case, one can show again that
    \[TV(\theta(t))-TV(\theta(0))=2C_{\hat{J},\hat{J}}t,\]
    for small enough $t$.
\end{itemize}
Thus, combining all the cases, we conclude that $TV(\theta(t))$ is an affine map of $t$ for small enough $t>0$ and 
\begin{equation}\label{eq:tv-deriv-full}
    \frac{d}{dt}\left(\lambda\,TV(\theta(t))\right)=2\lambda\,C_{\hat{J},\hat{J}}.
\end{equation}
Therefore, combining \eqref{eq:data-deriv-full} and \eqref{eq:tv-deriv-full}, we obtain
\[\phi'_{+}(0)=-\tau|\hat{J}|+\#\{j\in \hat{J}:\hat\theta^{\max}_j\ge y_j\}+2\lambda C_{\hat{J},\hat{J}}\leq -u_{\hat{J},\hat{J}}+|\hat{J}|\leq 0,\]
where the last inequality follows from the assumption that $u_{\hat{J},\hat{J}}\geq |\hat{J}|.$ However, since $\hat\theta^{\max}$ minimizes $F$, $\phi'_+(0)\ge 0$,
hence $\phi'_+(0)=0$.
However, we have already argued that $\phi$ is affine on $[0,\eta_1]$ for some small enough $\eta_1>0$. Therefore, it must hold that $\phi(t)=\phi(0)$ for $t\in[0,\eta_1]$. Choose $\varepsilon\in(0,\eta_1]$ and define
$\tilde\theta=\hat\theta^{\max}+\varepsilon\one_I$.
Then $\tilde\theta\in\mathcal \cS$ and
$\tilde\theta_i>\hat\theta_i^{\max}$,
contradiction! Thus, the assumption of $u_{\hat{J},\hat{J}}\geq |\hat{J}|$ must be wrong and the proof of Lemma~\ref{lem:J_hat_in_tilde_I} is thus complete.
\end{proof}
\refstepcounter{section}
\section*{Appendix \thesection: Auxiliary Results Used in Section~\ref{sec:monotonicity}}\label{appnb}
\begin{lemma}\label{lem:qtl_linearity}
Define
\[Q_{\tau}(\theta):=\sum_{i=1}^{n}\rho_{\tau}(y_i-\theta_i).\]
    For $0\leq\tau_1,\tau_2\leq 1$, the following holds for any fixed $\theta \in \bR^n$:
    \[Q_{\tau_2}(\theta)-Q_{\tau_1}(\theta)=(\tau_2-\tau_1)\sum_{i=1}^n (y_i-\theta_i).\]
\end{lemma}
\begin{proof}
    \begin{align*}
        &\rho_{\tau_2}(y_i-\theta_i)-\rho_{\tau_1}(y_i-\theta_i)\\
        &=\begin{cases}
            (\tau_2-1)(y_i-\theta_i)-(\tau_1-1)(y_i-\theta_i) & \text{ if }y_i<\theta_i;\\
            \tau_2(y_i-\theta_i)-\tau_1(y_i-\theta_i) & \text{ if }y_i\geq\theta_i.
        \end{cases}\\
        &=(\tau_2-\tau_1)(y_i-\theta_i).
    \end{align*}
    Thus, the proof is complete by using the definition of $Q_{\tau}.$
\end{proof}
\begin{proof}[of Lemma~\ref{lemma:general_submodular}]
   Since submodularity is preserved under non-negative linear combinations, and we are given $w_{ij} \geq 0$, it suffices to prove that each individual term in the sum is submodular. That is, for any fixed pair $(i,j)$ with $i \neq j$ and any convex function $\phi$, we must show that $g(\theta) = \phi(\theta_i - \theta_j)$ is submodular. 

Let $x, y \in \bR^n$. We need to show:
\begin{equation*}
    \phi(x_i - x_j) + \phi(y_i - y_j) \geq \phi(\min(x_i, y_i) - \min(x_j, y_j)) + \phi(\max(x_i, y_i) - \max(x_j, y_j)).
\end{equation*}
To simplify notation, let $a_1 = x_i$, $a_2 = y_i$, $b_1 = x_j$, and $b_2 = y_j$. The inequality becomes:
\begin{equation}\label{eqn:submodularity_convex}
    \phi(a_1 - b_1) + \phi(a_2 - b_2) \geq \phi(\min(a_1, a_2) - \min(b_1, b_2)) + \phi(\max(a_1, a_2) - \max(b_1, b_2)).
\end{equation}
We proceed by analyzing the possible orderings of the pairs $(a_1, a_2)$ and $(b_1, b_2)$.
\\
\noindent \textbf{Case 1: }
Assume $(a_1 - a_2)(b_1 - b_2) \geq 0$. Without loss of generality, assume $a_1 \leq a_2$ and $b_1 \leq b_2$.
Then, $\min(a_1, a_2) = a_1$, $\min(b_1, b_2) = b_1$, $\max(a_1, a_2) = a_2$, and $\max(b_1, b_2) = b_2$.
Evaluating the right-hand side of \eqref{eqn:submodularity_convex}, we get:
\begin{equation*}
    \phi(a_1 - b_1) + \phi(a_2 - b_2).
\end{equation*}
This is exactly equal to the left-hand side. Thus, the inequality \eqref{eqn:submodularity_convex} holds with equality. 
\\
\noindent 
\textbf{Case 2:}
Assume $(a_1 - a_2)(b_1 - b_2) < 0$. Without loss of generality, assume $a_1 < a_2$ and $b_1 > b_2$. 
Then, $\min(a_1, a_2) = a_1$, $\max(a_1, a_2) = a_2$, $\min(b_1, b_2) = b_2$, and $\max(b_1, b_2) = b_1$.
The showing \eqref{eqn:submodularity_convex} becomes equivalent to showing:
\begin{equation} \label{eq:target_ineq}
    \phi(a_1 - b_1) + \phi(a_2 - b_2) \geq \phi(a_1 - b_2) + \phi(a_2 - b_1).
\end{equation}
To prove this using the convexity of $\phi$, let us define 
\begin{align*}
    &V = a_1 - b_1, \quad
    U = a_2 - b_2, \\
    &X = a_2 - b_1, \quad
    Y = a_1 - b_2.
\end{align*}
Notice that $X + Y = (a_2 - b_1) + (a_1 - b_2) = a_1 + a_2 - b_1 - b_2$, and $U + V = (a_2 - b_2) + (a_1 - b_1) = a_1 + a_2 - b_1 - b_2$. Thus, $X + Y = U + V$.

Since $a_1 < a_2$ and $b_2 < b_1$, we can strictly order these points:
\begin{align*}
    V &= a_1 - b_1 < a_2 - b_1 = X, \\
    V &= a_1 - b_1 < a_1 - b_2 = Y, \\
    X &= a_2 - b_1 < a_2 - b_2 = U, \\
    Y &= a_1 - b_2 < a_2 - b_2 = U.
\end{align*}
Therefore, both $X$ and $Y$ lie strictly in the interval $(V, U)$. We can express $X$ and $Y$ as convex combinations of $U$ and $V$. Let $t \in (0, 1)$ be defined as:
\begin{equation*}
    t = \frac{a_2 - a_1}{(a_2 - a_1) + (b_1 - b_2)}.
\end{equation*}
Note that $t > 0$ because $a_2 > a_1$ and $b_1 > b_2$. It is straightforward to verify that:
\begin{equation*}
    t U + (1-t) V = \frac{a_2 - a_1}{(a_2 - a_1) + (b_1 - b_2)}(a_2 - b_2) + \frac{b_1 - b_2}{(a_2 - a_1) + (b_1 - b_2)}(a_1 - b_1) = a_2 - b_1 = X.
\end{equation*}
By symmetry, since $X + Y = U + V$, we have:
\begin{equation*}
    (1-t) U + t V = Y.
\end{equation*}
Because $\phi$ is a convex function, we apply the definition of convexity to $X$ and $Y$:
\begin{align*}
    \phi(X) &= \phi(t U + (1-t) V) \leq t \phi(U) + (1-t) \phi(V), \\
    \phi(Y) &= \phi((1-t) U + t V) \leq (1-t) \phi(U) + t \phi(V).
\end{align*}
Adding these two inequalities together yields:
\begin{align*}
    \phi(X) + \phi(Y) &\leq [t + (1-t)] \phi(U) + [(1-t) + t] \phi(V) \\
    \phi(X) + \phi(Y) &\leq \phi(U) + \phi(V).
\end{align*}
Substituting the original variables back into the inequality, we get:
\begin{equation*}
    \phi(a_2 - b_1) + \phi(a_1 - b_2) \leq \phi(a_2 - b_2) + \phi(a_1 - b_1),
\end{equation*}
which is exactly the required inequality \eqref{eq:target_ineq}. Thus, the proof is complete.
\end{proof}
\refstepcounter{section}
\section*{Appendix \thesection: Pointwise Risk Bound for Quantile TVD}\label{appnc}
    \begin{proposition}\label{prop:ptwise_error_bound}
 Under the quantile sequence model with some $\tau \in (0,1)$, suppose Assumption \ref{ass:growth_condn} holds. Given any $c>1$, let $$C:=\frac{(c+3)\tau}{8c_1} \text{ and }C_1:=\frac{c+2}{2c_1^2\delta^2}.$$
   Then, there exists a large enough universal constant $\tilde{C}>0$ and a natural number $N_0$ (depends on $c,\, c_1,\,\tau$) such that for any $n\geq N_0$ and $C\log n\leq \lambda \leq n$, the following bounds hold for the estimator defined in \eqref{eqn:objective_fn} at \textit{all location} $i\in [n]$ \textit{simultaneously} with probability $\geq 1-4n^{-(c-1)}:$
    \[\tilde{L}_i\leq \hat{\theta}_i-\theta_i\leq \tilde{U}_i,\]
    where
    \begin{align*}
\tilde{L}_i &:= 
\begin{cases}
\displaystyle \max_{\substack{J \in \mathcal{I}_i: J \subseteq [2:n-1] \\ |J| > 4\lambda/(c_1\delta),\, \text{Dist}(i, \partial J) \ge C_1 \log n}} f\left(i,\,\tau,\, J,\, \theta^*,\,\lambda\right) & \text{if } i=\lceil C_1 \log n\rceil,\ldots, \lfloor n - C_1 \log n\rfloor; \\[2.5em]
\displaystyle \max_{\substack{J \in \mathcal{I}_i: J = [1:j_2] \\ |J| > 4\lambda/(c_1\delta),\, j_2 - i + 1 \ge C_1 \log n}} f\left(i,\,\tau,\, J,\, \theta^*,\,\lambda\right) & \text{if } i=1,\ldots,\lceil C_1 \log n\rceil; \\[2.5em]
\displaystyle \max_{\substack{J \in \mathcal{I}_i: J = [j_1:n] \\ |J| > 4\lambda/(c_1\delta),\, i - j_1 + 1 \ge C_1 \log n}} f\left(i,\,\tau,\, J,\, \theta^*,\,\lambda\right) & \text{if } i=\lfloor n - C_1 \log n \rfloor +1,\ldots,n;
\end{cases}
\end{align*}
\begin{align*}
\tilde{U}_i &:= 
\begin{cases}
\displaystyle \min_{\substack{J \in \mathcal{I}_i: J \subseteq [2:n-1] \\ |J| > 4\lambda/(c_1\delta),\, \text{Dist}(i, \partial J) \ge C_1 \log n}} g\left(i,\,\tau,\, J,\, \theta^*,\,\lambda\right) & \text{if } i=\lceil C_1 \log n\rceil,\ldots, \lfloor n - C_1 \log n\rfloor; \\[2.5em]
\displaystyle \min_{\substack{J \in \mathcal{I}_i: J = [1:j_2] \\ |J| > 4\lambda/(c_1\delta),\, j_2 - i + 1 \ge C_1 \log n}} g\left(i,\,\tau,\, J,\, \theta^*,\,\lambda\right) & \text{if } i=1,\ldots,\lceil C_1 \log n\rceil; \\[2.5em]
\displaystyle \min_{\substack{J \in \mathcal{I}_i: J = [j_1:n] \\ |J| > 4\lambda/(c_1\delta),\, i - j_1 + 1 \ge C_1 \log n}} g\left(i,\,\tau,\, J,\, \theta^*,\,\lambda\right) & \text{if } i=\lfloor n - C_1 \log n \rfloor +1,\ldots,n,
\end{cases}
\end{align*}
where $\mathcal{I}_i$ is the set of discrete sub-intervals of $[1:n]$ containing $i$, and $f,g$ are defined as
\[f\left(i,\,\tau,\, J,\, \theta^*,\,\lambda\right):=Bias_-(i,J,\theta^*)-SD^{1-\tau}(i,J,\lambda),\]
\[g\left(i,\,\tau,\, J,\, \theta^*,\,\lambda\right):=Bias_+(i,J,\theta^*)+SD^{\tau}(i,J,\lambda).\]
\end{proposition}
\begin{proof}[of Proposition~\ref{prop:ptwise_error_bound}]
Note that~\eqref{eqn:bias_sd} implies that for any $J \in \I_i$,
\[\hat{\theta}_i-\theta^*\leq \underbrace{\max_{k \in J} (\theta^*_k-\theta^*_i)}_{Bias_{+}(i,J,\theta^*)}+\max_{I\subseteq J: i \in I}\,\epsilon_{I,(\lfloor u_{I,J}\rfloor+1)}.\]
We will consider some particular choices for $J$ depending on the location $i \in [n]$.\\
\\
\noindent \underline{$(a)\;i\in [\lceil C_1 \log n\rceil:\lfloor n-C_1\log n\rfloor]$, where $C_1>0$ is suitable large enough constant:}\\
\\
\noindent In this case, we only consider those $J\subseteq [2:n-1]$, such that $|J|>4\lambda/(c_1\delta)$ and $Dist(i,\partial J)\geq C_1 \log n$.
Note that
\begin{align*}
    \max_{I\subseteq J: i \in I}\,\epsilon_{I,(\lfloor u_{I,J}\rfloor+1)}&\leq \max_{I\subseteq J: i \in I,\,I\cap \partial J\neq \emptyset}\,\epsilon_{I,(\lfloor \tau |I|\rfloor+1)} 
    +\max_{I\subseteq J: i \in I,\;I\cap \partial J= \emptyset}\,\epsilon_{I,(\lfloor \tau |I|-2\lambda\rfloor+1)}
    +\epsilon_{I,(\lfloor \tau |J|+2\lambda\rfloor+1)}\\
    &= \max_{I\subseteq J: i \in I,\;I\cap \partial J\neq \emptyset}\,\epsilon_{I,(\lfloor \tau |I|\rfloor+1)} 
    +\max_{I\subseteq J: i \in I,\;I\cap \partial J= \emptyset}\,\epsilon_{I,(\lfloor \tau |I|-2\lambda\rfloor+1)}
    +\epsilon_{J,(\lfloor \tau |J|+2\lambda\rfloor+1)}.
\end{align*}
Now, for $t>0$, and $I\subseteq J$,
\begin{align}
    &\bP\left(\epsilon_{I,(\lfloor u_{I,J}\rfloor+1)}>t\right)\nonumber\\
    &=\bP\left(\sum_{k \in I} 1(\epsilon_k>t)\geq |I|-\lfloor u_{I,J}\rfloor\right)\nonumber\\
    &\leq \bP\left(\sum_{k \in I} 1(\epsilon_k>t)\geq |I|- u_{I,J}\right)\nonumber\\
    &=\bP\left(\sum_{k \in I} 1(\epsilon_k>t)\geq (1-\tau)|I|+2\lambda C_{I,J}\right)\nonumber\allowdisplaybreaks\\
    &=\bP\left(\sum_{k \in I} 1(\epsilon_k>t)-\sum_{k \in I} \Bar{F}_k(t)\geq (1-\tau)|I|+2\lambda C_{I,J}-\sum_{k \in I} \Bar{F}_k(t)\right)\nonumber\\
    &=\bP\left(\frac{1}{|I|}\sum_{k \in I} (1(\epsilon_k>t)-\Bar{F}_k(t))\geq \frac{2\lambda}{|I|} C_{I,J}+\frac{1}{|I|}\sum_{k \in I} (F_k(t)-\tau)\right)\label{eqn:tail_prob_bound}.
\end{align}
When $C_{I,J}=1$, i.e. $I \subset J$, and $0<t\leq \delta$, by Hoeffding inequality,~\eqref{eqn:tail_prob_bound} becomes 
\[\leq \exp{\left(-2|I|(2\lambda/|I|+\sum_{k \in I} (F_k(t)-\tau)/|I|)^2\right)}\leq\exp{\left(-2|I|[2\lambda/|I| +c_1t]^2\right)}\leq \exp{\left(-8\lambda c_1t/\tau\right)},\]
where the second inequality follows from Assumption~\ref{ass:growth_condn} and the third one follows from the inequality $(a+b)^2\geq 2ab$.\\
\noindent When $C_{I,J}=0$, i.e. $I$ shares one endpoint with $J$, if $0<t\leq \delta$, using Assumption~\ref{ass:growth_condn} and Hoeffding inequality, we can bound ~\eqref{eqn:tail_prob_bound} by
\begin{equation}\label{eqn:hpbd_boundary}
    \exp{\left(-2c_1^2|I|t^2\right)}\leq \exp{\left(-2c_1^2\,Dist(i,\partial J)\,t^2\right)}.
\end{equation}
Also, note that there are exactly $|J|$ many such $I$.
 Therefore, using union bound, for $0<t_1,t_2,t_3\leq \delta$, we have
\begin{align}\label{eqn:hpbd_risk}
    &\bP\left((\lfloor u_{I,J}\rfloor+1)-\text{th order statistic of }\epsilon_I>t_1+t_2+t_3\right)\nonumber\\
    &\leq |J| \exp{\left(-2c_1^2\,Dist(i,\partial J)\,t_1^2\right)}+\sum_{I:i\in I\subseteq J,\,I\cap \partial J= \emptyset}\exp{\left(-8\lambda c_1t/\tau\right)}\nonumber\allowdisplaybreaks\\
    &+\exp{\left(-2|J|\left[\frac{1}{|J|}\sum_{k \in J} (F_k(t_3)-\tau)-2\lambda\right]^2\right)}\nonumber\\
    &\leq |J| \exp{\left(-2c_1^2\,Dist(i,\partial J)\,t_1^2\right)}+|J|^2\exp{\left(-8\lambda c_1t/\tau\right)}\nonumber\\
    &+\bP\left(\frac{1}{|J|}\sum_{k \in J} (1(\epsilon_k>t_3)-\Bar{F}_k(t_3))\geq -\frac{2\lambda}{|J|} +\frac{1}{|J|}\sum_{k \in J} (F_k(t_3)-\tau)\right).
\end{align}
\noindent 
Now, we set $t_1=C_2 \left(\log n/Dist(i,\partial J)\right)^{1/2}$, where $C_2=\left((c+2)/(2c_1^2)\right)^{1/2}$ depends only on $c_1$ and $c$. Note that under the assumption $Dist(i,\partial\,J)\geq C_1 \log n$, where $C_1= \left((c+2)/(2c_1^2\delta^2)\right)$, $0<t_1\leq \delta$. Moreover, the first term in~\eqref{eqn:hpbd_risk} becomes \(\leq n^{-(c+1)}.\)\\
\noindent Next, we set $t_2=C_3\tau\log n/\lambda$, where $C_3=\left((c+3)/(8c_1)\right)$. Thus, the second term in \eqref{eqn:hpbd_risk} becomes \(\leq n^{-(c+1)}\). Note that choosing $\lambda \geq C \log n$ where $C=\left((c+3)\tau/(8c_1)\right)$, the above choice of $t_2$ can be made $\leq \delta.$\\
\noindent Finally, when $C_{I,J}=-1$, i.e. $I=J$, we set
\[t_3:=\frac{2\lambda}{c_1|J|}+\left(\frac{c+1}{2c_1^2}\frac{\log n}{ |J|}\right)^{1/2}.\]
Now, we assumed in this case  that \(|J|>4\lambda/c_1\delta\). Then, we have \(2\lambda/(c_1|J|)\leq \delta/2\). Moreover, the second term is already $\leq \delta/2$ because $C_1= \frac{c+2}{2c_1^2\delta^2}\geq \frac{c+1}{4c_1^2\delta^2}$. Therefore, $0<t_3\leq \delta$. Now, under Assumption~\ref{ass:growth_condn}, we have
    \[\frac{1}{|J|}\sum_{k \in J} (F_k(t_3)-\tau)-\frac{2\lambda}{|J|}\geq \frac{2\lambda}{|J|}+\left((c+1)\frac{\log n}{2|J|}\right)^{1/2}-\frac{2\lambda}{|J|}=\left((c+1)\frac{\log n}{2|J|}\right)^{1/2}>0.\]
    Thus, by Hoeffding inequality, we can bound~\eqref{eqn:tail_prob_bound} in this case by
\[\exp{\left(-2(c+1)|J|\frac{\log n}{2|J|}\right)}=n^{-(c+1)}.\]
 Now, since $Dist(i,\partial J)\leq |J|$,  we conclude that for such a $J$ containing $i$, the following bound holds with probability $\geq 1-3n^{-c}:$
 \[ \max_{I\subseteq J: i \in I}\,\epsilon_{I,(\lfloor u_{I,J}\rfloor+1)}\leq SD^{\tau}(i,J,\lambda),\]
 for a large enough universal constant $\tilde{C}>0.$ Finally, we point out that with the above choices of the constants, at least one $J$ always exists satisfying the required conditions if we take $n$ to be large enough. Thus, the upper bound is not vacuous.
 \\
 \\
\noindent \underline{$(b)\;i\in [1:\lceil C_1 \log n\rceil]$, where $C_1>0$ is suitable large enough constant:}\\
\\
\noindent In this case, consider $J$ of the form $[1:j_2]$, where $j_2\geq i$ and $j_2-i+1\geq C_1\, \log n.$ We would again like to show that with probability $\geq 1-4n^{-c}$, 
\[ \max_{I\subseteq J: i \in I}\,\epsilon_{I,(\lfloor u_{I,J}\rfloor+1)}\leq SD^{\tau}(i,J,\lambda),\]
for a large enough universal constant $\tilde{C}>0.$ By Definition~\ref{defn:C_{IJ}}, \(C_{I,J}\) takes four possible values here. Note that the cases when $C_{I,J}\neq 0$ can be handled in a similar fashion as we did in case $(a)$. So, we only focus on the subcase where $C_{I,J}=0$. This happens when $I=[s:j_2]$, $1<s\leq i .$ Using \eqref{eqn:hpbd_risk}, similar to \eqref{eqn:hpbd_boundary}, we would again get
\[\bP\left(\epsilon_{I,(\lfloor u_{I,J}\rfloor+1)}>t\right)\leq  \exp{\left(-2c_1^2|I|t^2\right)}\leq  \exp{\left(-2c_1^2(j_2-i+1)t^2\right)},\]
if $0<t\leq \delta.$  There are exactly $i$ many such $I$. Thus, we need to find $0<t\leq \delta$ such that
\[n\exp{\left(-2c_1^2(j_2-i+1) t^2\right)}\leq n^{-(c+1)}.\]
Solving for $t$, we get $t= \left((c+2)\log n/(2c_1^2(j_2-i+1))\right)^{1/2}.$
Now, since $j_2-i+1=Dist(i,\partial J)\geq C_1 \log n$ with $C_1=(c+2)/(2c_1^2\delta^2)$, the above $t$ can be made smaller than $\delta$. Thus, the upper bound in Case $(b)$ is established.\\
 \\
\noindent \underline{$(c)\;i\in [\lfloor n-C_1 \log n\rfloor+1:n]$, where $C_1>0$ is suitable large enough constant:}\\
\\
\noindent The proof in this case is very similar to Case $(b)$.\\
\\
\noindent Now, going back to \eqref{eqn:bias_sd}, we have
\begin{align*}
    &\hat{\theta}_i-\theta^*_i\\
    &\leq  \min_{J \in \I_i}\left[Bias_+(i,J,\theta^*)+\max_{I\subseteq J: i \in I}\,\epsilon_{I,(\lfloor u_{I,J}\rfloor+1)}\right]\allowdisplaybreaks\\
    &\leq \one \{i\in [\lceil C_1\,\log n\rceil:\,\lfloor n-C_1\,\log n\rfloor]\}\, \min_{\substack{J \in \I_i:J\subseteq [2:n-1]\\ |J|>4\lambda/(c_1\delta),\,Dist(i,\partial\,J)\geq C_1 \log n\,}}\left[Bias_+(i,J,\theta^*)+\max_{I\subseteq J: i \in I}\,\epsilon_{I,(\lfloor u_{I,J}\rfloor+1)}\right]\allowdisplaybreaks\\
    &+\one \{i\in [1:\lceil C_1\,\log n\rceil]\}\, \min_{\substack{J \in \I_i:J= [1:j_2]\\ |J|>4\lambda/(c_1\delta),\,j_2-i+1\geq C_1 \log n\,}}\left[Bias_+(i,J,\theta^*)+\max_{I\subseteq J: i \in I}\,\epsilon_{I,(\lfloor u_{I,J}\rfloor+1)}\right]\\
    &+\one \{i\in [\lfloor n-C_1\,\log n\rfloor:n]\}\, \min_{\substack{J \in \I_i:|J|>4\lambda/(c_1\delta),\,J= [j_1:n]\\ |J|>4\lambda/(c_1\delta),\,i-j_1+1\geq C_1 \log n\,}}\left[Bias_+(i,J,\theta^*)+\max_{I\subseteq J: i \in I}\,\epsilon_{I,(\lfloor u_{I,J}\rfloor+1)}\right]\\
     &\leq \one \{i\in [\lceil C_1\,\log n\rceil:\,\lfloor n-C_1\,\log n\rfloor]\}\, \min_{\substack{J \in \I_i: J\subseteq [2:n-1]\\ |J|>4\lambda/(c_1\delta),\,Dist(i,\partial\,J)\geq C_1 \log n\,}}\left[Bias_+(i,J,\theta^*)+SD^{\tau}(i,J,\lambda)\right]\\
    &+\one \{i\in [1:\lceil C_1\,\log n\rceil]\}\, \min_{\substack{J \in \I_i:J= [1:j_2]\\ |J|>4\lambda/(c_1\delta),\;j_2-i+1\geq C_1 \log n\,}}\left[Bias_+(i,J,\theta^*)+SD^{\tau}(i,J,\lambda)\right]\\
    &+\one \{i\in [\lfloor n-C_1\,\log n\rfloor:n]\}\, \min_{\substack{J \in \I_i:J= [j_1:n]\\ |J|>4\lambda/(c_1\delta),,\;i-j_1+1\geq C_1 \log n\,}}\left[Bias_+(i,J,\theta^*)+SD^{\tau}(i,J,\lambda)\right],\\
\end{align*}
where the last bound holds with probability $\geq 1-4 n^{-c}$ by using Case $(a)-(c)$ and union bound over different choices of $J$. Thus, the upper bound in the proposition holds simultaneously for $i\in [n]$ with probability $\geq 1-4n^{-(c-1)}$ by applying union bound. \\
\\
For the lower bound on $\hat{\theta}_{i}-\theta^*_i$, we resort to the inequality in \eqref{eqn:bias_sd_2}. Note that it suffices to establish high probability lower bound on the stochastic term involved in \eqref{eqn:bias_sd_2}. Now, for any $t>0$,
\begin{align*}
    \min_{I\subseteq J: i \in I}\,\epsilon_{I,(\lceil l_{I,J}\rceil)}\leq -t&\Rightarrow -\min_{I\subseteq J: i \in I}\,\epsilon_{I,(\lceil l_{I,J}\rceil)}\geq t\\
    &\Rightarrow \max_{I\subseteq J: i \in I}\,-\epsilon_{I,(\lceil l_{I,J}\rceil)}\geq t\\
    &\Rightarrow \max_{I\subseteq J: i \in I}\,(-\epsilon)_{I,(|I|-\lceil l_{I,J}\rceil+1)}\geq t,
\end{align*}
where the last line follows from \eqref{eqn:order_stat_neg_vector}. Now, we argue that $|I|-\lceil l^{\tau}_{I,J}\rceil=\lfloor |I|-l^{\tau}_{I,J}\rfloor=\lfloor u^{1-\tau}_{I,J}\rfloor.$ The second equality is obvious. To see why the first equality holds, let $k$ be an integer such that $k<l_{I,J}\leq k+1$. Thus, $\lceil l_{I,J}\rceil=k+1$. Moreover, $|I|-k-1\leq |I|-l_{I,J}<|I|-k$ and since $|I|$ is an integer, we conclude that $\lfloor |I|- l_{I,J}\rfloor=|I|-k-1$. Thus,
\[\min_{I\subseteq J: i \in I}\,\epsilon_{I,(\lceil l_{I,J}\rceil)}\leq -t\Rightarrow \max_{I\subseteq J: i \in I}\,(-\epsilon)_{I,(\lfloor \lfloor u^{1-\tau}_{I,J}\rfloor)}\geq t.\]
Thus, this is equivalent to proving high probability upper bound on the pointwise estimation error, with the quantile level being $1-\tau$ and the errors being $-\epsilon_k$ s. Since CDF of $-\epsilon_k$ is $G_k(t)=1-F_k(-t),$ Assumption~\ref{ass:growth_condn} holds for $G_k$ s with same constants $c_1$ and $\delta$. Thus, one can now complete the proof by using the same argument we used in deriving the upper bound.
 \end{proof}
\end{appendix}
\end{document}